\documentclass[]{interact}

\usepackage{graphicx}%
\usepackage{subfigure}
\usepackage{multirow}%
\usepackage{amsmath,amssymb,amsfonts}%
\usepackage{amsthm}%
\usepackage{mathrsfs}%
\usepackage[title]{appendix}%
\usepackage{xcolor}%
\usepackage{textcomp}%
\usepackage{manyfoot}%
\usepackage{booktabs}%
\usepackage{algorithm}%
\usepackage{algorithmicx}%
\usepackage{algpseudocode}%
\usepackage{listings}%
\usepackage{CJK}
\usepackage{indentfirst}
\usepackage{cases}
\usepackage{longtable}
\usepackage{array}
\usepackage{relsize}

\newcommand{\Diag}{\hbox{Diag}}

\usepackage{soul} 
\usepackage{color, xcolor} 

\usepackage{epstopdf}
\usepackage[caption=false]{subfig}

\usepackage[numbers,sort&compress]{natbib}
\bibpunct[, ]{[}{]}{,}{n}{,}{,}
\makeatletter
\def\NAT@def@citea{\def\@citea{\NAT@separator}}
\makeatother

\theoremstyle{plain}
\newtheorem{theorem}{Theorem}[section]
\newtheorem{lemma}[theorem]{Lemma}

\newtheorem{proposition}[theorem]{Proposition}

\theoremstyle{definition}
\newtheorem{definition}[theorem]{Definition}

\newtheorem{assumption}[theorem]{Assumption}

\theoremstyle{remark}

\begin{document}
	
	
\title{An Efficient Smoothing Damped Newton Method for Large-Scale Mathematical Programs with Equilibrium Constraints}
	
\author{
\name{Yixin Wang\textsuperscript{a}\thanks{Email: eathelynwang@outlook.com} and 
Qingna Li\textsuperscript{a,b}\thanks{Email: qnl@bit.edu.cn. Corresponding author. This author’s research is supported by the National Science Foundation of China (NSFC) 12071032 and 12271526.} and 
Liwei Zhang\textsuperscript{c}\thanks{Email: zhanglw@mail.neu.wdu.cn. This author’s research is supported by the National Natural Science Foundation of China (NSFC) 12371298.}}
\affil{\textsuperscript{a}School of Mathematics and Statistics, Beijing Institute of Technology, Beijing, China; \textsuperscript{b}Beijing Key Laboratory on MCAACI/Key Laboratory of Mathematical Theory and Computation in Information Security, Beijing Institute of Technology, Beijing, China; \textsuperscript{c}National Frontiers Science center for Industrial Intelligence and Systems Optimization, Northeastern University, 110819, Liaoning, China.}
	}
	
\maketitle
	
\begin{abstract}
Bilevel hyperparameter optimization has received growing attention thanks to the fast development of machine learning. Due to the tremendous size of data sets, the scale of bilevel hyperparameter optimization problem could be extremely large, posing great challenges in designing efficient numerical algorithms. In this paper, we focus on solving the large-scale mathematical programs with equilibrium constraints (MPEC) derived from hyperparameter selection of L1 support vector classification (L1-SVC). We propose a highly efficient smoothing damped Newton method (SDNM) for solving such MPEC. Compared with most existing algorithms where subproblems are solved by packages, our approach fully takes advantage of the structure of MPEC and therefore is package-free. Moreover, the proposed SDNM converges to C-stationary point under MPEC-LICQ with subproblem enjoys a quadratic convergence rate under proper assumptions. Extensive numerical results over LIBSVM dataset show the superior performance of SDNM over other state-of-art algorithms.
\end{abstract}
	

\section{Introduction}

Bilevel hyperparameter optimization has received growing attention thanks to the fast development of machine learning. Due to the tremendous size of data sets, the scale of bilevel hyperparameter optimization problem could be extremely large, posing great challenges in designing efficient numerical algorithms. Having been one of the most active research areas, vast literature concerns bilevel hyperparameter optimization (BHO) in machine learning. We refer to \cite{dempe2002foundations,petrulionyte2024functional,yu2020hyper,zhang2024introduction}
for surveys and monographs. Below we focus on numerical algorithms for BHO in support vector machine (SVM), which is most relevant to our research in this paper.

We start with an early work by Mangasarian \cite{Mangasarian1994}, which provides a foundation for an equivalent reformulation (linear-programming problem with complementarity constraints, LPCC) of the so-called counting loss $\|(\cdot)_+\|_0$. 
The story begins with Bennett et al.'s work \cite{Bennett_01}, where they proposed a bilevel optimization framework for hyperparameter selection problem in SVM. By replacing the lower-level problem by KKT conditions to get a single-level problem (which is a mathematical program with equilibrium constraints, MPEC), Bennett et al. proposed a successive linearization method to solve such MPEC,  with subproblem solved by CPLEX. 
Meanwhile, with the equivalent LPCC reformulation of counting loss function, Kunapuli et al. \cite{Kunapuli_02} studied BHO for $L_1$-regularized support vector classification ($L_1$-SVC), with counting loss and $L_1$ loss as upper-level and lower-level functions, respectively. They used FILTER package to solve the resulting MPEC. 
Later work includes rectangle search algorithm by Lee et al. \cite{Lee_03} to solve the resulting single-level LPCC with subproblem solved by CPLEX, stochastic gradient method by Couellan and Wang \cite{Couellan_04} for the resulting single-level MPEC (no subproblem involved), and Scholtes global relaxation method \cite{Li_07,Li_05,Li_06} for MPEC with subproblem solved by SNOPT \cite{snopt}.

Different from above existing work, Wang and Li \cite{Wang_08} studied BHO for Logistic-loss SVC with Logistic loss as upper-level and lower-level loss function. Different from BHO for $L_1$-SVC and $L_2$-SVC, the resulting single-level problem is a nonlinear programming problem (NLP) with only equality constraints. A smoothing Newton method is proposed to solve the KKT conditions of such NLP, which does not involve subproblem. Very recently, Ye et al. \cite{Ye_09} proposed a difference-of-convex algorithm for BHO in $L_1$-SVM with $L_1$-loss upper-level function, with subproblem solved by SDPT3. Similar idea is also studied in \cite{Ye_10,Ye_11}. Qian et al. \cite{Li2023GlobalRL} proposed an $L_p$-Newton method to solve the single MPEC in \cite{Kunapuli_02}, with subproblem solved by Gurobi.

To summarize, we obtain the following observations. 
(i) Traditional single-level problem derived from BHO for SVM is based on replacing lower problem by KKT conditions, and leads to an MPEC.
(ii) Most algorithms for solving such MPEC are two-loop, where subproblems are often solved by software packages. However, calling software packages has several drawbacks. Firstly, it ignores the structure of the subproblem, therefore, may not be as efficient as specially designed algorithms for subproblems and may take longer time in solving subproblems. 
Secondly, recalling the large-scale of MPEC due to the huge size of data sets, the speed of the algorithms for BHO in SVM is still far from satisfactory. For example, based on our computational experience in \cite{Wang_08}, the CPU time for the Scholtes global relaxation method (SGRM) in \cite{Li_05} is over one hour when the number of features of the dataset is over 100.

Based on the above analysis, together with our successful experience in designing smoothing Newton method to solve NLP derived from BHO for Logistic-loss SVC, a natural question is whether it is possible to design a specialized algorithm to solve MPEC derived from BHO for $L_1$-SVC. This motivates the work in this paper. 
The contribution of the paper is summarized in three folds. 
Firstly, we propose a smoothing method to solve large-scale MPEC derived from BHO for $L_1$-SVC and the global convergence is established under reasonable assumptions.
Secondly, to solve the subproblem, we design the damped Newton method, and analyze the first and second order optimality conditions. Under certain assumptions, we show the global convergence and local quadratic convergence rate.
Finally, we conduct extensive numerical experiments to verify the efficiency of the proposed smoothing damped Newton method, especially the superior performance in speed.

The paper is organized as follows. 
In Section 2, we briefly restate the BHO for $L_1$-SVC and the resulting single-level MPEC in \cite{Li_05}, which is the problem that we will solve in this paper. 
In Section 3, we introduce the smoothing approach to deal with complementarity constraints in MPEC and analyze the convergence of the smoothing method.
In Section 4, we analyze the perturbation property of the feasible set of MPEC.
In Section 5, we propose the damped Newton method to solve the subproblem, and establish its convergence results under certain assumptions.
In Section 6, we conduct extensive numerical experiments to verify the high efficiency of the proposed smoothing damped Newton method (SDNM), especially the superior performance in speed.
Finally, we conclude the paper in Section 7.

\vskip 2mm

\noindent {\bf Notations.} 
We use $\mathcal{I}$ to denote the identity matrix.
We use Diag$(a_1,\cdots,a_n)$ to denote the diagonal matrix with diagonal blocks $a_1,\cdots,a_n$.
Let $[n]$ denote $\{1,\cdots,n\}$ for positive integer $n$.

\section{Large-Scale MPEC: A Motivating Example}\label{sec2}

In this paper, we focus on the following MPEC:
\begin{equation*} \label{pb_mpec}\tag{MPEC}
	\begin{array}{cl}
		\min\limits_{ v \in\mathbb{R}^{m+1}}& f( v )\\
		\hbox{s.t.}& G(v)\ge0,\ H(v)\ge 0,\ G(v)^\top H(v)=0.
	\end{array}
\end{equation*}
where $f: \mathbb{R}^{m+1} \to \mathbb{R}$, $G: \mathbb{R}^{m+1} \to \mathbb{R}^{m}$, $H: \mathbb{R}^{m+1} \to \mathbb{R}^{m}$ are affine functions. Such form of \eqref{pb_mpec} is motivated by the single-level reformulation of BHO for $L_1$-SVC, which we will briefly introduce below. More details can be found in \cite{Li_05}. 


The bilevel model for hyperparameter selection in $L_1$-SVC is based on cross-validation (CV). Consider a dataset $\Omega \in \left\lbrace \left( x_i, y_i\right) \right\rbrace _{i=1}^{p_1} \subseteq \mathbb{R}^{n+1}$, where $x_i\in \mathbb{R}^n$ denotes a data point and $y_i\in\left\lbrace \pm 1 \right\rbrace $ is the corresponding label. CV evaluates the performance of $L_1$-SVC by dividing the data set into $T$ folds of equal size, denote as $\Omega_t,\ t\in[T]$. 
In the $t$-th lower-level problem, given $C \geq 0$, we train the $t$-th fold training set $\overline{\Omega}_{t}:=\Omega \backslash\Omega_t$ by the soft-margin SVC model with the $L_1$-loss function\footnote{We choose the $L_1$-SVC model as the typical lower-level problem due to the following reasons. Firstly, from the practical perspective, the $L_1$-SVC model is a widely used statistical model in machine learning \cite{dedieu2022solving,yan2020efficient,zheng2023l1}. 
Secondly, the $L_1$-SVC model is more challenging to tackle than the $L_2$-SVC model due to the nonsmoothness of the $L_1$-loss function.}.
The aim of the upper-level problem is to minimize the $T$-fold cross-validation error (CV error) measured on the validation sets based on the optimal solutions of the lower-level problems. Specifically, the basic bilevel optimization model for selecting the hyperparameter $C$ in $L_1$-SVC is formulated as
\begin{equation}\small \label{pb_bilevel}
\begin{array}{cl}
\min \limits_{C \in \mathbb{R},\ w^{t} \in \mathbb{R}^{n},\ t=1, \cdots,T} & \frac{1}{T}\sum\limits_{t=1}^{T} \frac{1}{m_{1}} \sum\limits_{i \in \mathcal{N}_{t}} \| \max\left(0, -y_{i} \left( x_{i}^{\top} w^{t}\right) \right)\|_{0}\\
\hbox{s.t.} \ & C \geq 0,\\
& \text{and for} \ t=1,\cdots,T: \\
& \qquad w^{t} \in \underset{ w \in \mathbb{R}^{n}}{\operatorname{argmin}} \left\{\frac{1}{2}\|w\|_{2}^{2}+ C \sum\limits_{i \in \overline{\mathcal{N}_{t}}}  \max \left( 0,1-y_{i}\left(x_{i}^{\top}w\right)\right)
\right\},
\end{array}
\end{equation}
where $\mathcal{N}_{t}$ and $\overline{\mathcal{N}_{t}}$ denote the index set of $\Omega_t$ and $\overline{\Omega_t}$ respectively, and $m_1$ and $m_2$ are the numbers of data points in $\Omega_t$ and $\overline{\Omega_t}$. 
Here, the expression $\| \max\left(0, -y_{i} \left( x_{i}^{\top} w^{t}\right) \right)\|_{0}$ basically counts the number of data points that are misclassified in the validation set $\Omega_{t}$, while the outer summation (i.e., the objective function in \eqref{pb_bilevel}) averages the misclassification error over all the folds.

Note that each lower-level problem is strongly convex, one can replace the lower-level problem by its KKT conditions after introducing slack variable $\zeta^t\in \mathbb{R}^{m_2},\ t = 1,\cdots,T$. For the upper level objective function in \eqref{pb_bilevel}, one can introduce the technique in \cite{Mangasarian1994} to rephrase it as a linear programming problem with complementarity constraints\footnote{To make the problem well-defined, we assume that $x_i^\top w^t\neq 0$, for all $i\in\mathcal{N}_t$ amd $t\in [T]$.}. The above two strategies lead to the following MPEC reformulation of \eqref{pb_bilevel} (see \cite{Li_05} for more details),
\begin{equation} \label{pb_mpec_ori}
	\begin{array}{cl}
		\min\limits_{\begin{subarray}{c}
				C \in \mathbb{R} \\
				\zeta \in \mathbb{R}^{Tm_{1}},\ z \in \mathbb{R}^{Tm_{1}}\\
				\alpha \in \mathbb{R}^{Tm_{2}},\ \xi \in \mathbb{R}^{Tm_{2}}
		\end{subarray}}  &\frac{1}{Tm_{1}} \mathbf{1}^{\top} \zeta\\
		\hbox{s.t.}& \mathbf{0} \leq  \zeta \perp  A B^{\top} \alpha+z  \geq \mathbf{0}, \\
			&\mathbf{0} \leq  z  \perp  \mathbf{1}-\zeta  \geq \mathbf{0},  \\
			&\mathbf{0} \leq  \alpha  \perp  B B^{\top} \alpha- \mathbf{1}+\xi \geq \mathbf{0},\\
			&\mathbf{0} \leq  \xi  \perp  C \mathbf{1}-\alpha  \geq \mathbf{0},
	\end{array}
\end{equation}
where $\textbf{1}$ is the vector with all elements one, and
\begin{equation*}\small
\begin{array}{c}
A:={\rm Diag} \left( A^1,\ A^2,\cdots,A^n\right) \in\mathbb{R}^{Tm_1\times Tn},\
B:={\rm Diag} \left( B^1,\ B^2,\cdots,B^n\right) \in\mathbb{R}^{Tm_1\times Tn},
\\ 
A^t:=\left[\begin{array}{c}y_{t_1}x_{t_1}^\top\\\vdots\\y_{t_{m_1}}x_{t_{m_1}}^\top\end{array}\right]\in\mathbb{R}^{m_1\times n},\ B^t:=\begin{bmatrix}y_{t_{m_1+1}}x_{t_{m_1+1}}^\top\\\vdots\\y_{t_{p_1}}x_{t_{p_1}}^\top\end{bmatrix}\in\mathbb{R}^{m_2\times n},\ t\in[T],
\\ 
{\mathsmaller{\zeta:=\left[\begin{array}{l} \zeta^{1} \\ \zeta^{2} \\ \vdots \\ \zeta^{T}\end{array}\right] \in\mathbb{R}^{Tm_{1}},\  
	z:= \left[\begin{array}{l} z^{1} \\ z^{2} \\ \vdots \\ z^{T}\end{array}\right]\in\mathbb{R}^{Tm_{1}},\ \alpha:=\left[\begin{array}{l}\alpha^1\\\alpha^2\\\vdots\\\alpha^T\end{array}\right]\in\mathbb{R}^{Tm_{2}},\
	\xi:=\left[\begin{array}{l}\xi^1\\\xi^2\\\vdots\\\xi^T\end{array}\right]\in\mathbb{R}^{Tm_{2}}.}}
\end{array}
\end{equation*}

Therefore, \eqref{pb_mpec_ori} can be written into a compact form of \eqref{pb_mpec}, where $m = 2T(m_1+m_2),\  f( v ):= \frac{1}{Tm_{1}} \mathbf{1}^{\top} \zeta$, and $G( v )\in \mathbb{R}^m,\ H( v )\in \mathbb{R}^m,\ v\in\mathbb{R}^{m+1}$ are defined by
\begin{equation} \label{def_G_H_v} 
	G( v ):=\left[\begin{array}{c} \zeta\\ z\\ \alpha\\ \xi \end{array}\right],\ 
	H( v ):=\left[\begin{array}{c} AB^\top\alpha+z\\ \mathbf{1}-\zeta\\ BB^\top\alpha-\mathbf{1}+\xi\\ C\mathbf{1}-\alpha \end{array}\right],\ 
	 v :=\left[\begin{array}{c}C\\ \zeta\\ z\\ \alpha\\ \xi \end{array}\right].
\end{equation}

Specifically, one can notice that the number of components in $v$ and the number of complementarity constraints are $m+1$ and $m$, respectively, where $m=2T(m_1+m_2)$, $m_1$ and $m_2$ are the number of training data and validation data, respectively. When the number of cross validation or the number of dataset increases, it will lead to the increase of the scale of \eqref{pb_mpec}. Solving such a large-scale bilevel hyperparameter selection problem for SVC is extremely challenging, in which the cross-validation process increases the scale of the problem exponentially.

We close this section by the following definitions of constraint qualification and stationary point of \eqref{pb_mpec}. Let $v^*$ be a feasible point for \eqref{pb_mpec}, define the following index sets:
\begin{equation*}\begin{array}{ll}
I_{0+}(v^*) &= \{i \in [m] : G_i(v^*) = 0, \, H_i(v^*) > 0\}, \\
I_{+0}(v^*) &= \{i \in [m] : G_i(v^*) > 0, \, H_i(v^*) = 0\}, \\
I_{00}(v^*) &= \{i \in [m] : G_i(v^*) = 0, \, H_i(v^*) = 0\}.
\end{array} \end{equation*}

We usually use $I_{0+}$ instead of $I_{0+}(v^*)$ when the dependence of $I_{0+}$ on $v^*$ is very clear. It is the same with $I_{+0}$ and $I_{00}$.

\begin{definition} \cite[page 920]{jiani30} (MPEC-LICQ) A feasible point $v^*$ for problem \eqref{pb_mpec} satisfies the MPEC-LICQ if and only if the set of gradient vectors
\begin{equation*}
\left\{ \nabla G_i(v^*) \mid i \in I_{0+} 
\cup I_{00} \right\} \cup \left\{ \nabla H_i(v^*) \mid i \in I_{+0} \cup I_{00} \right\} 
\end{equation*}
is linearly independent.
\end{definition}

\begin{definition}
Let $v^*$ be feasible for \eqref{pb_mpec}. Then $v^*$ is said to be\\
(a) weakly stationary, if there are multipliers $\gamma$, $\nu \in \mathbb{R}^m$ such that
\begin{equation*}
    \nabla f(v^*) - \sum_{i=1}^m \gamma_i \nabla G_i(v^*)- \sum_{i=1}^m \nu_i \nabla H_i(v^*) = 0
\end{equation*}
and $\gamma_i = 0\ (i \in I_{+0})$, $\nu_i = 0\ (i \in I_{0+})$;\\
(b) C-stationary, if it is weakly stationary and $\gamma_{i} \nu_{i} \geq 0$ for all $i \in I_{00}$.
\end{definition}

\section{Solution Method}
To solve \eqref{pb_mpec}, traditional methods include relaxation method \cite{Hoheisel2013,kadrani2009new,kanzow2013new,lin2005modified,Scholtes2001,steffensen2010new}, duality method \cite{li2023novel,li2024solving}, exact penalization approach \cite{scholtes1999exact}, stochastic method \cite{patriksson1999stochastic,shapiro2008stochastic}. We refer to \cite{dempe2003annotated} for review solving \eqref{pb_mpec}.
Keeping in mind that our aim is to design an efficient numerical algorithm to solve \eqref{pb_mpec}, which is software-independent, the complementarity constraints are the first issue that we want to tackle. 
Smoothing technique has been widely used in solving many problems such as sparse optimization problems \cite{bian2020smoothing,bian2023accelerated} and nonsmooth optimization problems \cite{liang2024squared}. 
In this section, we solve \eqref{pb_mpec} by introducing smoothing functions to deal with the complementarity constraints in \eqref{pb_mpec}. We will also investigate the relation of the smoothing problem and \eqref{pb_mpec}.

In \eqref{pb_mpec}, $0\le H( v )\perp G( v )\ge0$ are linear complementarity constraints. Therefore, from the optimization point of view, there are various methods to tackle complementarity constraints to their smoothing equivalent system to make the computation process more efficient \cite{gharbia2023semismooth,luo1996mathematical,nguyen2025smoothing,ochs2016techniques}. 
Different smoothing functions have been proposed and studied \cite{chen2000penalized,chen2007some,FB0}. 
Here we choose the popular Fischer-Burmeister (FB) smoothing function \cite{FB0}, which is defined by 
\begin{equation*}
   \phi_\epsilon(a,b)=a+b-\sqrt{a^2+b^2+\epsilon^2},\quad a,b\in\mathbb{R}. 
\end{equation*}
$\phi_\epsilon(a,b)=0$ if and only if $a>0,b>0, ab=\frac{\epsilon^2}{2}$.
For $s\in\mathbb{R}^m,\ t\in\mathbb{R}^m$, define 
\begin{equation*}
\Phi_\epsilon(s,t)=\left[\phi_\epsilon(s_1,t_1),\ \cdots, \ \phi_\epsilon(s_m,t_m) \right]^\top\in\mathbb{R}^m.
\end{equation*}

The FB smoothing function has many nice properties useful in algorithmic development for semidefinite complementarity problems and nonsmooth equations involving the semidefinite conic complementarity condition. It is known from \cite{FB1} that this function is strongly semismooth, which plays a fundamental role in the analysis of the quadratic convergence of Newton-type methods for solving such nonsmooth equation systems, see \cite{FB2} for instance. Furthermore, it is proved by \cite{FB3} that the gradient mapping of the squared norm of the FB smoothing function is Lipschitz continuous.

Applying $\Phi(\cdot,\cdot)$ to the complementarity constraints in \eqref{pb_mpec}, we reach the smoothing version of \eqref{pb_mpec} as follows ($\epsilon>0$), which is a essentially a nonlinear programming problem \eqref{pb_nlp} with equality constraints:
\begin{equation}\label{pb_nlp}\tag{\text{NLP$_\epsilon$}}
\begin{array}{cl}
	\min\limits_{ v \in\mathbb{R}^{m+1}}& f( v )\\
	\hbox{s.t.}& \Phi_\epsilon(G( v ), H( v ))=0.
\end{array}
\end{equation}

Here we would like to make a comparison between \eqref{pb_mpec} and \eqref{pb_nlp}. In terms of variables, the two problems have the same scale of variables, which is $m+1$. However, the number of constraints in \eqref{pb_nlp} is only one third of that in \eqref{pb_mpec}, which is much smaller. Moreover, the constraints in \eqref{pb_nlp} are equality constraints, which are much easier to deal with in designing algorithms.

Given $\{\epsilon_t\}\searrow0$, we solve (NLP$_{\epsilon_t}$) in each iteration. We reach the following smoothing method in Algorithm \ref{algo1} to solve \eqref{pb_mpec}.

\begin{algorithm}[htbp]
\caption{Smoothing Method for \eqref{pb_mpec}} \label{algo1}
\begin{algorithmic}[1]

\State Choose $\epsilon_0\ge\epsilon_{\min}>0$, $\kappa\in(0,1)$, $t:=0$.

\State \textbf{while} $\epsilon_t\ge\epsilon_{\min}$, \textbf{do}

\State \quad Solve (NLP$_{\epsilon_t}$) to get $v^{t+1}.$

\State \quad Let $\epsilon_{t+1}:=\kappa\epsilon_t$,  and $t:=t+1$. 

\State \textbf{end while}

\State \textbf{Return} the final iterate $v_{\rm opt}:=v^{t}$.

\end{algorithmic}
\end{algorithm}

Here we would like to highlight the advantage of the proposed algorithm. Firstly, it is software independent, which do not rely on other packages. Additionally, by smoothing the complementarity constraints, the number of constraints is reduced to one third, which is much smaller than that solved by other state-of-art methods, and also contributes to the speed up of our algorithm.

The following theorem shows that under the fulfillment of MPEC-LICQ, any accumulation point of the smoothing algorithm is a C-stationary point of \eqref{pb_mpec}.

\begin{theorem}\label{thm_C}
    Let $\{ \epsilon_t \} \searrow 0$ and let $v^{t}$ be a stationary point of (NLP$_{\epsilon_t}$) with $v^*$ be any accumulation point such that MPEC-LICQ holds at $v^*$. If LICQ holds at $v^{t}$ for each $t$, then $v^*$ is a C-stationary point of \eqref{pb_mpec}.
\end{theorem}

\begin{proof}
Since $v^{t}$ is a KKT point of (NLP$_{\epsilon_t}$) and LICQ holds at $v^{t}$ for each $t$, there exists unique multiplier $\lambda$ such that the following holds
\begin{equation}\label{eq_kkt_pf}
\begin{array}{ll}
\textbf{0} 
&= \nabla f(v^{t}) - \sum_{i=1}^m \lambda^{(t)}_i \nabla \phi_{\epsilon_t}(G_i(v^{t}),H_i(v^{t}))\\
&= \nabla f(v^{t})-\sum_{i=1}^m \lambda^{(t)}_i \left( w^G_i (v^{t}) \nabla G_i(v^{t}) +w^H_i(v^{t}) \nabla H_i(v^{t})\right),
\end{array}
\end{equation}
where $i\in[m]$, and $\forall \epsilon_t>0$,
\begin{equation*}
\begin{array}{c}
w^G_i(v^{t})=1-\dfrac{G_i(v^{t})}{\sqrt{G^2_i(v^{t}) +H_i^2(v^{t})+(\epsilon_t)^2}}\in(0,1],\\
w^H_i(v^{t})=1-\dfrac{H(v^{t})}{\sqrt{G^2_i(v^{t}) +H_i^2(v^{t})+(\epsilon_t)^2}}\in(0,1].
\end{array} 
\end{equation*}
Together with $\{\epsilon_t \}\searrow 0$, and that $v^*$ is an accumulation point of $v^{t}$, by taking subsequence if necessary, assume that $w^G_i(v^{t})\to a_i$, $w^H_i(v^{t})\to b_i$. By defining 
\begin{equation*}
\begin{array}{c}
I_{00}^{1G} = \{i\in I_{00}\ |\ a_i=0 \},\ I_{00}^{2G} = \{i\in I_{00}\ |\ a_i=1 \},\ I_{00}^{3G} = \{i\in I_{00}\ |\ a_i\in(0,1) \},\\
I_{00}^{1H} = \{i\in I_{00}\ |\ b_i=0 \},\ I_{00}^{2H} = \{i\in I_{00}\ |\ b_i=1 \},\ I_{00}^{3H} = \{i\in I_{00}\ |\ b_i\in(0,1) \}.
\end{array}
\end{equation*}
it holds that
\begin{equation*}
\begin{array}{cc}
 a_i\in\begin{cases}
    \{0\},\ i\in I_{+0},\\
    \{1\},\ i\in I_{0+},\\
    \{0\},\ i\in I_{00}^{1G},\\
    \{1\},\ i\in I_{00}^{2G},\\
    (0,1),\ i\in I_{00}^{3G},
\end{cases}
& b_i\in
\begin{cases}
    \{1\},\ i\in I_{+0},\\
    \{0\},\ i\in I_{0+},\\
    \{0\},\ i\in I_{00}^{1H},\\
    \{1\},\ i\in I_{00}^{2H},\\
    (0,1),\ i\in I_{00}^{3H}.
\end{cases}
\end{array} 
\end{equation*}

Assume for contradiction that the sequence $\left\{\lambda^{(t)}\right\}$ is unbounded, then one can find a subsequence $K$ such that the $\frac{\lambda^{(t)}}{\|\lambda^{(t)}\|}$ converges to some vector, denoted as $\hat{\lambda} \neq 0$.

Denote $\hat{I}_+:=\{i\mid \hat{\lambda}_i>0 \}$, $\hat{I}_-:=\{i\mid \hat{\lambda}_i<0 \}$, 
$\hat{I}_0:=\{i\mid \hat{\lambda}_i=0 \}$.
By taking subsequence if necessary, it holds that
\begin{equation*}
\dfrac{\lambda^{(t)}}{\|\lambda^{(t)}\|}  w^G_i(v^{t}) \to \hat{\gamma_i} = \hat{\lambda}_i a_i,\ \dfrac{\lambda^{(t)}}{\|\lambda^{(t)}\|}  w^H_i(v^{t}) \to \hat{\nu_i} = \hat{\lambda}_i b_i.
\end{equation*}
where
\begin{equation*}
\begin{array}{cc}
\hat{\gamma_i}=
\begin{cases}
    \hat{\lambda}_i>0,&i\in \hat{I}_+\cap (I_{0+}\cup I_{00}^{2G}),\\
    \hat{\lambda}_i a_i>0,&i\in  \hat{I}_+\cap I_{00}^{3G},\\
    \hat{\lambda}_i<0,&i\in \hat{I}_-\cap (I_{0+}\cup I_{00}^{2G}),\\
    \hat{\lambda}_i a_i<0,&i\in  \hat{I}_-\cap I_{00}^{3G},\\
    0,&i\in I_{+0}\cup I_{00}^{1G}\cup \hat{I}_{0},
\end{cases}&
\hat{\nu_i}=
\begin{cases}
    \hat{\lambda}_i>0,&i\in \hat{I}_+\cap (I_{+0}\cup I_{00}^{2H}),\\
    \hat{\lambda}_i b_i>0,&i\in  \hat{I}_+\cap I_{00}^{3H},\\
    \hat{\lambda}_i<0,&i\in \hat{I}_-\cap (I_{+0}\cup I_{00}^{2H}),\\
    \hat{\lambda}_i b_i<0,&i\in  \hat{I}_-\cap I_{00}^{3H},\\
    0,&i\in I_{0+}\cup I_{00}^{1H}\cup \hat{I}_{0}.
\end{cases}
\end{array}
\end{equation*}
Therefore, taking limits for $t \to +\infty$ in \eqref{eq_kkt_pf} gives
\begin{equation*}
0=-\sum_{i\in I_{0+}\cup I_{00}^{2G}\cup I_{00}^{3G}} \hat{\gamma_i} \nabla G_i(v^*)-\sum_{i\in I_{+0}\cup I_{00}^{2H}\cup I_{00}^{3H}} \hat{\nu_i}\nabla H_i(v^*),
\end{equation*}
where $\hat{\gamma_i}\neq0, i \in i\in I_{0+}\cup I_{00}^{2G}\cup I_{00}^{3G}; \hat{\nu_i}\neq0, i\in I_{+0}\cup I_{00}^{2H}\cup I_{00}^{3H}$. 
Therefore, it is a contradiction to the prerequisite that MPEC-LICQ holds at $v^*$, since $I_{00}^{2G}\cup I_{00}^{3G} \subseteq I_{00}$ and $I_{00}^{2H}\cup I_{00}^{3H} \subseteq I_{00}$. Thus, we may assume without loss of generality that one of the accumulation points of $\left\{\lambda^{k}\right\}$ is $\lambda^*$. The continuous differentiability of $f, G, H$ implies
\begin{equation}\label{eq_kkt_ws}
    \textbf{0}= \nabla f(v^*)-\sum_{i=1}^m \gamma_i^*\nabla G_i(v^*)-\sum_{i=1}^m \nu_i^*\nabla H_i(v^*),
\end{equation}
where $\gamma_i^{*} = \lambda^{*}_i a_i,$, $\nu_i^{*} = \lambda^{*}_i b_i$.

By defining $I_+^*$, $I_-^*$ and $I_0^*$ on $\lambda^*$ similar to that of $\hat{I}_+$, $\hat{I}_-$ and $\hat{I}_0$, one can obtain the following 
\begin{equation*}
\begin{array}{c}
\gamma_i^*=
\begin{cases}
    \lambda_i^*>0,&i\in I_+^*\cap (I_{0+}\cup I_{00}^{2G}),\\
    \lambda_i^* a_i>0,&i\in  I_+^*\cap I_{00}^{3G},\\
    \lambda_i^*<0,&i\in I_-^*\cap (I_{0+}\cup I_{00}^{2G}),\\
    \lambda_i^* a_i<0,&i\in  I_-^*\cap I_{00}^{3G},\\
    0,&i\in I_{+0}\cup I_{00}^{1G} \cup I_0^*,
\end{cases}\ 
\nu_i^*=
\begin{cases}
    \lambda_i^*>0,&i\in I_+^*\cap (I_{+0}\cup I_{00}^{2H}),\\
    \lambda_i^* b_i>0,&i\in  I_+^*\cap I_{00}^{3H},\\
    \lambda_i^*<0,&i\in I_-^*\cap (I_{+0}\cup I_{00}^{2H}),\\
    \lambda_i^* b_i<0,&i\in  I_-^*\cap I_{00}^{3H},\\
    0,&i\in I_{0+}\cup I_{00}^{1H} \cup I_0^*.
\end{cases}
\end{array}
\end{equation*}
It holds that
\begin{equation*}
0\le \gamma_i^* \nu_i^* =
\begin{cases}
    (\lambda_i^*)^2,& i\in (I_+^*\cup I_-^*) \cap (I_{0+} \cup I_{00}^{2G}) \cap (I_{+0} \cup I_{00}^{2H}),\\
    (\lambda_i^*)^2 a_i,&i\in (I_+^*\cup I_-^*) \cap I_{00}^{3G} \cap (I_{+0} \cup I_{00}^{2H}),\\
    (\lambda_i^*)^2 b_i,& i\in (I_+^*\cup I_-^*) \cap I_{00}^{3H} \cap (I_{0+} \cup I_{00}^{2G}),\\
    (\lambda_i^*)^2 a_i b_i,&i\in (I_+^*\cup I_-^*) \cap  I_{00}^{3G}\cap I_{00}^{3H},\\
    0,& {\rm otherwise.}
\end{cases}
\end{equation*}
 
Overall, based on the definition of C-stationary point, $v^*$ is a C-stationary point of \eqref{pb_mpec}. The proof is finished.
\end{proof}

Based on Theorem \ref{thm_C}, to guarantee the convergence to a C-stationary point,we need the MPEC-LICQ at accumulation point $v^*$, as well as the LICQ at $v^{t}$ for each $t$. The MPEC-LICQ is guaranteed when $\left |  I_{GH}\right |=1$ by the following theorem proposed in \cite{lijiani25}. The fulfillment of LICQ at each $v^{t}$ will be addressed in Section 5.

\begin{theorem}\label{LICQ_new}
Let  $v^*$  be a feasible point of problem \eqref{pb_mpec}.
\begin{itemize}
\item[$\left(i\right)$] If  $\left | I_{GH} \right | > 1 $, then MPEC-LICQ fails at $ v^* $.
\item[$\left(ii\right)$] If  $\left | I_{GH} \right | =0 $, then MPEC-LICQ holds at $ v^* $.
\item[$\left(iii\right)$] If $\left |  I_{GH}\right |=1$,   $(BB^{\top})_{(\Lambda_{3}^+, \Lambda_3^+)}$ is positive definite at $v^*$ and $\Gamma_{sub} \neq 0$, then MPEC-LICQ holds at $ v^* $, where
\begin{equation*}\footnotesize 
    \Gamma_{\mathrm{sub}}=\begin{bmatrix}(BB^\top)_{(I_{GH_4},\Lambda_{3}^c\cup\Lambda_{u})}\,\mathbf 1
     - (BB^\top)_{(I_{GH_4},\Lambda_{3}^+)}\,
    (BB^\top)_{(\Lambda_{3}^+,\Lambda_{3}^+)}^{-1}\,
    (BB^\top)_{(\Lambda_{3}^+,\Lambda_{3}^c\cup\Lambda_{u})}\,\mathbf 1\\
    (BB^\top)_{(I_{GH_3},\Lambda_{3}^c\cup\Lambda_{u})}\,\mathbf 1
    - (BB^\top)_{(I_{GH_3},\Lambda_{3}^+)}\,
    (BB^\top)_{(\Lambda_{3}^+,\Lambda_{3}^+)}^{-1}\,
    (BB^\top)_{(\Lambda_{3}^+,\Lambda_{3}^c\cup\Lambda_{u})}\,\mathbf 1
    \end{bmatrix}
\end{equation*} 	
and 
\begin{equation*}
    \begin{cases}
        \Lambda_{3}^{+}:=\{i\in Q_{l}\ |\ 0<\alpha_{i}<C,\ (BB^{\top}\alpha-\mathbf{1}+\xi)_{i}=0,\ \xi_{i}=0\},\\
        \Lambda_{3}^{c}:=\left\{i\in Q_{l}\ |\ \alpha_{i}=C,(BB^{\top}\alpha-\mathbf{1}+\xi)_{i}=0,\ \xi_{i}=0\right\},\\
        \Lambda_u:=\left\{i\in Q_l\ |\ \alpha_i=C,\ (BB^\top\alpha-\mathbf{1}+\xi)_i=0, \ \xi_i>0\right\},\\
        I_{GH_{3}}:=\{i\in Q_{l}\ |\ \alpha_{i}=0,\ (BB^{\top}\alpha-\mathbf{1}+\xi)_{i}=0\},\\
        I_{GH_4}:=\{i\in Q_l\ |\ \xi_i=0,\ C-\alpha_i=0\},\\        
        Q_l:=\{1,2,\cdots,Tm_2\}.
    \end{cases}
\end{equation*}
\end{itemize}
\end{theorem}

\section{Feasible Sets of \eqref{pb_nlp} and \eqref{pb_mpec}}

Before focusing on solving the smoothing subproblem \eqref{pb_nlp}, we need to investigate the feasible region of \eqref{pb_nlp} and \eqref{pb_mpec}. Due to the smoothing function in  \eqref{pb_nlp} and the positive smoothing parameter $\epsilon$, the feasible region of \eqref{pb_nlp} is no longer the same as that of \eqref{pb_mpec}. Therefore, what is the relation of feasible sets of \eqref{pb_nlp} and \eqref{pb_mpec}, as $\epsilon\searrow 0$? We address this question below.

Denote the feasible sets of \eqref{pb_mpec} and \eqref{pb_nlp} by
\begin{equation*}\Omega(\epsilon) =\{ v\in \mathbb{R}^{m+1} \mid \Phi_\epsilon(H( v ), G( v ))=0\},\ \Omega^0 =\{ v\in \mathbb{R}^{m+1} \mid 0\le G( v )\perp H( v )\ge0\}.\end{equation*}

Consider the following equivalent form of \eqref{pb_mpec} and \eqref{pb_nlp}
\begin{align}
&\begin{array}{cl}\label{pb_mpec_mat}
\min\limits_{ v \in\mathbb{R}^{m+1},p\in\mathbb{R}^{m}, q\in\mathbb{R}^m}&  f( v )\\	
\hbox{s.t.}& G( v )=p, \ H( v )=q,\ 0\le p\perp q\ge 0,
\end{array}\\
&\begin{array}{cl}\label{pb_nlp_mat}
\min\limits_{ v \in\mathbb{R}^{m+1},p\in\mathbb{R}^{m}, q\in\mathbb{R}^m}& f( v )
	\\	
	\hbox{s.t.}&  G( v )=p, \ H( v )=q,\ \Phi_\epsilon(p, q)=0.
\end{array}
\end{align}
Denote the feasible sets of \eqref{pb_mpec_mat} and \eqref{pb_nlp_mat} as $\Gamma^0$ and $\Gamma(\epsilon)$, which are given by 
\begin{equation} \label{def_gamma}
	\begin{aligned}
		&\Gamma^0:=\Gamma^=\cap \Gamma^\perp,\ \Gamma(\epsilon):=\Gamma^=\cap \Gamma^\perp(\epsilon),\ \Gamma^= :=\{( v ,p,q)\mid G( v )=p, \ H( v )=q\},\\
		&\Gamma^\perp:=\{( v ,p,q)\mid  0\le p\perp q\ge 0\},\ \Gamma^\perp(\epsilon):=\{( v ,p,q)\mid \Phi_\epsilon(p,q)=0\}.
	\end{aligned}
\end{equation}
We first address the relationship between $\underset{\epsilon\searrow 0}{\lim}\sup\Gamma(\epsilon)$ and $\Gamma^0$ in the following lemma.

\begin{lemma}\label{lem_1}
	Let $\Gamma(\epsilon)$ and $\Gamma^0$ be defined in \eqref{def_gamma}, it holds that $\underset{\epsilon\searrow 0}{\lim}\sup\Gamma(\epsilon)\subseteq \Gamma^0.$
\end{lemma}

\begin{proof}
For any $( v ,p,q)\in 	\underset{\epsilon\searrow 0}{\lim}\sup\Gamma(\epsilon)$, by definition, there exists $\epsilon_k\searrow 0$ and $( v^k,p^k,q^k)\to ( v ,p,q)$ with $( v^k,p^k,q^k)\in\Gamma(\epsilon_k)$. Next, we need to show that $( v ,p,q)\in\Gamma^0$.
To that end, by taking limits, we obtain that
\begin{equation*}G( v )=\underset{k\to\epsilon}\lim G( v^k) =\underset{k\to\epsilon}\lim p^k=p,\ H( v )=\underset{k\to\epsilon}\lim H( v^k) =\underset{k\to\epsilon}\lim q^k=q.\end{equation*}
Moreover, we have  $\underset{k\to\epsilon}\lim \Phi_\epsilon(p^k, q^k) =\underset{k\to\epsilon}\lim \frac{(\epsilon_k)^2}{2}=0,$
implying that $ 0\le G( v )\perp H( v )\ge 0$. This gives that $( v ,p,q)\in\Gamma^0$. The proof is finished.
\end{proof}

To show that $\Gamma^0\subseteq \underset{\epsilon\searrow 0}{\lim}\inf\Gamma(\epsilon)$, we define the following sets for each $( v ,p,q)\in\Gamma^0$: $I_{0+}^\Gamma=\{i\in m\mid p_i=0,\ q_i>0\}$, $I_{+0}^\Gamma=\{i\in m\mid p_i>0,\ q_i=0\}$, $I_{00}^\Gamma=\{i\in m\mid p_i=0,\ q_i=0\}$, and define $\theta_\epsilon(\iota)=\left( \theta_{\epsilon,1}(\iota_1), \theta_{\epsilon,2}(\iota_2), \cdots, \theta_{\epsilon,m}(\iota_m)\right) ^\top$ by
\begin{equation*} \label{eq_theta}
	\theta_{\epsilon,j}(\iota_j):=\begin{cases}
		\frac {\epsilon^2} {2\iota_j} - q_j, &\ j\in I_{0+}^\Gamma,\\
		\frac {\epsilon^2} {2\left( p_j + \iota_j \right) }, &\ j\in I_{+0}^\Gamma,\\
		\ \ \ \frac {\epsilon^2} {2\iota_j}, &\ j\in I_{00}^\Gamma.
	\end{cases}
\end{equation*}

The following assumption is also necessary.
\begin{assumption}\label{as_1}
	Given $( v ,p,q)\in\Gamma^0$, assume that for any  $\epsilon_k \searrow 0$, there exists $(\Delta C^k, \Delta p^k)$ such that 
	\begin{equation*} 
    \begin{cases}
        \Delta p^k\ge 0, \theta_{\epsilon_k}(\Delta p^k)\ge 0, (\Delta C^k, \Delta p^k,\theta_{\epsilon_k}(\Delta p^k))\to 0,\\
		L^H_{(:,1)}\Delta C^k+ L^H_{(:,2:n)}\Delta p^k=\theta_{\epsilon_k}(\Delta p^k).
    \end{cases}
	\end{equation*} 
\end{assumption}

With all these preparations, the result of $\underset{\epsilon\searrow 0}{\lim}\inf\Gamma(\epsilon)$ is given as follows.

\begin{lemma}\label{lem_2}
	Under Assumption \ref{as_1}, it holds that $\Gamma^0\subseteq \underset{\epsilon\searrow 0}{\lim}\inf\Gamma(\epsilon).$
\end{lemma}

\begin{proof}
To proof the result, we need to show that for any $( v ,p,q)\in \Gamma^0$, $( v ,p,q)\in \underset{\epsilon\searrow 0}{\lim}\inf\Gamma(\epsilon)$. That is, for any $\epsilon_k\searrow 0$, there exists a sequence $( v^k,p^k,q^k)$, such that $\lim\limits_{k\to\infty}( v^k,p^k,q^k)\to ( v ,p,q),\ {\rm and}\  ( v^k,p^k,q^k)\in\Gamma(\epsilon_k)$, which is,
\begin{equation}\label{eq_limvpq}
	\lim\limits_{k\to\infty}( v^k,p^k,q^k)= ( v ,p,q),\ H( v^k)=q^k,\ G( v^k)=p^k, \ \Phi_\epsilon(p^k, q^k)=\frac{(\epsilon_k)^2}{2}.
\end{equation}
Equivalently, by denoting $\Delta v^k=v^k-v,\ \Delta p^k=p^k-p,\ \Delta q^k=q^k-q$, \eqref{eq_limvpq} is equivalent to the following 
\begin{subnumcases}{} 
		H( v +\Delta v^k)=q+\Delta q^k,\\ 
		G( v +\Delta v^k)=p+\Delta p^k, \\
		\Phi_\epsilon(p+\Delta p^k, q+\Delta q^k)=\frac{(\epsilon_k)^2}{2}\label{eq_phi}\\
		(\Delta v^k,\Delta p^k,\Delta q^k)\to(0,0,0),\ \epsilon_k\to0,
\end{subnumcases}
where \eqref{eq_phi} is equivalent to the following 
\begin{equation*}\Delta p^k\ge0,\ \Delta q^k\ge0,\  (p_j+\Delta p^k_j)(q_j+\Delta q^k_j)=\frac{(\epsilon_k)^2}{2}.\end{equation*}
Recall the definition of $G(\cdot)$ and $H(\cdot)$ in \eqref{def_G_H_v}. $G(v)$ and $H(v)$ can be written as $G(v)=L^Gv$, and $H(v)=L^Hv+b^H$. The above system reduces to the following
\begin{equation}\label{eq_vpq_k}
	\begin{cases}
		L^H( v +\Delta v^k)+ b^H = q+\Delta q^k,\\ 
		L^G( v +\Delta v^k) = p+\Delta p^k, \\
		\Delta p^k\ge0,\ \Delta q^k\ge0,\  (p_j+\Delta p^k_j)(q_j+\Delta q^k_j)=\frac{(\epsilon_k)^2}{2},\\
		(\Delta v^k,\Delta p^k,\Delta q^k)\to(0,0,0),\ \epsilon_k\to0.
	\end{cases}
\end{equation}
Notice that $( v ,p,q)$ is a feasible point, it holds that
\begin{equation*}
	\begin{cases}
		L^H v +b^H=q,\\ 
		L^G v =p,\\
		p_{I_{0+}^\Gamma}=0,\ p_{I_{+0}^\Gamma}>0,\  p_{I_{00}^\Gamma}=0,\\  
		q_{I_{0+}^\Gamma}>0,\ q_{I_{+0}^\Gamma}=0,\  q_{I_{00}^\Gamma}=0.	
	\end{cases}
\end{equation*}
On the other hand, we have 
\begin{eqnarray*}
	(p_j+\Delta p^k_j)(q_j+\Delta q^k_j)&=&p_jq_j	+p_j\Delta q_j^k+q_j\Delta p_j^k+\Delta p_j^k\Delta q_j^k \\
	&=&p_j\Delta q_j^k+q_j\Delta p_j^k+\Delta p_j^k\Delta q_j^k\ \ (p_jq_j=0)\\
	&=&
	\begin{cases}
		q_j\Delta p_j^k+\Delta p_j^k\Delta q_j^k, &  j\in I_{0+}^\Gamma,\\
		p_j\Delta q_j^k+\Delta p_j^k\Delta q_j^k, &  j\in I_{+0}^\Gamma,\\
		\Delta p_j^k\Delta q_j^k, &  j\in I_{00}^\Gamma.
	\end{cases}
\end{eqnarray*}
The above system in \eqref{eq_vpq_k} reduces to the following form 
\begin{subnumcases}{}
	L^H\Delta v^k=\Delta q^k,\label{eq_vpq_k_re1}\\ 
	L^G\Delta v^k=\Delta p^k, \\
	q_j\Delta p_j^k+\Delta p_j^k\Delta q_j^k= \frac{(\epsilon_k)^2}{2},   j\in I_{0+}^\Gamma,\\
	p_j\Delta q_j^k+\Delta p_j^k\Delta q_j^k= \frac{(\epsilon_k)^2}{2},  j\in I_{+0}^\Gamma,\\
	\Delta p_j^k\Delta q_j^k=	 \frac{(\epsilon_k)^2}{2},   j\in I_{00}^\Gamma, \label{eq_vpq_k_re5}\\
	\Delta p^k\ge0,\ \Delta q^k\ge0,  (\Delta v^k,\Delta p^k,\Delta q^k)\to (0,0,0) ,\ \epsilon_k\to0.
\end{subnumcases}
By Assumption \ref{as_1}, for any $\epsilon\searrow 0$, there exists $(\Delta \widehat{C^k}, \Delta \widehat{p^k})$, such that  $\Delta \widehat{p^k}\ge 0$, $\theta_{\epsilon_k}(\Delta\widehat{p^k})$, $(\Delta\widehat{C^k}, \Delta \widehat{p^k},\theta_{\epsilon_k}(\widehat{p^k}))\to 0$ and 
$L^H_{(:,1)}\Delta C^k+ L^H_{(:,2:n)}\Delta p^k=\theta_{\epsilon_k}(\Delta p^k)$.
By letting $\Delta \widehat{q^k}=\theta_{\epsilon_k}(\Delta \widehat{p^k})$, 
$\left[\Delta \zeta^k \ \Delta z^k \ \Delta \alpha^k \ \Delta\xi^k \right]^\top=\Delta \widehat{p^k}$, one can verify that 
$\Delta\widehat{ v^k}=\left[\Delta C^k \ \Delta \zeta^k \ \Delta z^k \ \Delta \alpha^k \ \Delta\xi^k \right]^\top$, $\widehat{p^k},\ \widehat{q^k}$ satisfy \eqref{eq_vpq_k_re1}-\eqref{eq_vpq_k_re5}. Therefore, for any $\epsilon\searrow 0$, we found $(\Delta \widehat{ v^k}, \Delta \widehat{p^k}, \Delta \widehat{q^k})$ such that \eqref{eq_limvpq} holds. The proof is finished.
\end{proof}

Lemma \ref{lem_1} and Lemma \ref{lem_2} directly lead to the following result.

\begin{theorem} \label{thm_fea}
Under Assumption \ref{as_1}, it holds that $\underset{\epsilon\searrow 0}\lim \Gamma(\epsilon)=\Gamma^0$.	
\end{theorem}

This theorem demonstrates that under certain assumption, the feasible set of \eqref{pb_nlp} is equal to that of \eqref{pb_mpec}, as $\epsilon\searrow 0$.

\section{Solving Subproblem \eqref{pb_nlp}}
In this section, we consider the first and second order optimality conditions, and design a numerical algorithm to solve subproblem \eqref{pb_nlp} and study its convergence properties. 
\eqref{pb_nlp} is an optimization problem with nonlinear equality constraints, one way is to choose optimization packages, such as SNOPT or Matlab built-in function \textit{fmincon}. Keeping in mind that our aim is to design an efficient package-independent algorithm, together with the equality constraints in \eqref{pb_nlp}, so we decide to solve the KKT system \eqref{pb_kkt}, which is essentially a set of nonlinear equations.


\subsection{The linearly independence constraint qualification}

In this part, we will show that for smoothing problem \eqref{pb_nlp}, the linear independent constraint qualification (LICQ) holds at each feasible point $v$, which guarantees the uniqueness of the Lagrange multiplier. 

For simplicity, denote $W^G := \Diag(w^G_1,\dots, w^G_m),\ W^H := \Diag(w^H_1,\dots, w^H_m)$, where $w^G_i:=1-\frac{G_i( v )}{\sqrt{G^2_i( v )+H_i^2( v )+\epsilon^2}},\ w^H_i:=1-\frac{H_i( v )}{\sqrt{G^2_i( v )+H_i^2( v )+\epsilon^2}}, \ i\in[m].$
By applying the chain rule to calculate the Jacobi matrix of $\Phi_\epsilon(G( v ),H( v ))$, we obtain the following proposition.
\begin{proposition}\label{prop-JPhi}
	$J_ v \Phi_\epsilon(G( v ),H( v ))$ takes the following form
	\begin{equation*}
		J_ v \Phi_\epsilon(G( v ),H( v ))=W^GL^G+W^HL^H\in \mathbb{R}^{m\times(m+1)},
	\end{equation*}
    where
\begin{equation}\label{def_LG_LH_bH}\small
    L^G=\left[\begin{array}{ccccc}
	\mathbf{0}&\mathcal I&\mathbf{0}&\mathbf{0}&\mathbf{0}\\
	\mathbf{0}&\mathbf{0}&\mathcal I&\mathbf{0}&\mathbf{0}\\
	\mathbf{0}&\mathbf{0}&\mathbf{0}&\mathcal I&\mathbf{0}\\
	\mathbf{0}&\mathbf{0}&\mathbf{0}&\mathbf{0}&\mathcal I\\
\end{array}\right], \ L^H=\left[\begin{array}{ccccc}
	\mathbf{0}&\mathbf{0}&\mathcal I&AB^\top&\mathbf{0}\\
	\mathbf{0}&-\mathcal I&\mathbf{0}&\mathbf{0}&\mathbf{0}\\
	\mathbf{0}&\mathbf{0}&\mathbf{0}&BB^\top&\mathcal I\\
	\mathbf{1}&\mathbf{0}&\mathbf{0}&-\mathcal I&\mathbf{0}\\
\end{array}\right],\ b^H = \left[\begin{array}{c}
	\mathbf{0}\\
	\mathbf{1}\\
	\mathbf{-1}\\
	\mathbf{0}\\
\end{array}\right].
\end{equation}
\end{proposition}

\begin{proof}
	It is easy to verify that $J_ v  G( v ) = L^G, \ J_ v  H( v ) = L^H $. Also note that $\partial_a	\phi_\epsilon(a,b)=1-\frac{a}{\sqrt{a^2+b^2+\epsilon^2}}$. The result easily follows by the chain rule.
\end{proof}

Recall that in Theorem \ref{thm_C}, one of the assumptions is LICQ holds at $v^{t}$ for each $t$. Below we will show that LICQ property holds for each feasible point of \eqref{pb_nlp}. 

\begin{theorem}\label{thm_licq}
	LICQ holds at each feasible point $ v $ of \eqref{pb_nlp}.
\end{theorem}

\begin{proof}
We only need to show that, the Jacobian of $\Phi_\epsilon(H(v),G(v))$ with respect to $v$ (denoted as $J_v\Phi_\epsilon(H(v),G(v))$) has full row rank.
	
Let the diagonal matrix have the same partitions as $L^G$ and $L^H$. That is,
\begin{equation*} 
    W^G={\rm Diag}\left(W^G_1, W^G_2, W^G_3, W^G_4\right),\ 
    W^H={\rm Diag}\left(W^H_1, W^H_2, W^H_3, W^H_4\right),
\end{equation*}
where
$W^G_1,\ W^G_2,\ W^H_1,\ W^H_2 \in\mathbb{R}^{Tm_1\times Tm_1}, 
\ W^G_3,\ W^G_4,\ W^H_3,\ W^H_4 \in\mathbb{R}^{Tm_2\times Tm_2}.$
By Proposition \ref{prop-JPhi}, $J_ v \Phi_\epsilon(G( v ),H( v ))$ takes the following form
\begin{equation} \label{eq_Jphi} 
J_ v \Phi_\epsilon(G( v ),H( v ))=\left[\begin{array}{ccccc}
    {0}&W^G_1&W^H_1&W^H_1AB^\top&\mathbf{0}\\
    {0}&-W^H_2&W^G_2&\mathbf{0}&\mathbf{0}\\
    {0}&\mathbf{0}&\mathbf{0}&W^G_3+W^H_3BB^\top&W^H_3\\
    W^H_4\mathbf{1}&\mathbf{0}&\mathbf{0}&-W^H_4&W^G_4\\\end{array}\right].
\end{equation}
Assume that there exists nonzero vector $\rho$ with partition 
$\rho=\left[\rho_1^\top\ \rho_2^\top\ \rho_3^\top\ \rho_4^\top\right]^\top\in \mathbb{R}^m$
such that $\rho^\top J_ v \Phi_\epsilon(G( v ),H( v ))=0$. 
It gives the following results:

\begin{subnumcases}{}
\rho_4^\top W^H_4\mathbf{1}=0,\label{eq_Jphi_sys_1}\\
\rho_1^\top W^G_1-\rho_2^\top W^H_2=0,\label{eq_Jphi_sys_2}\\
\rho_1^\top W^H_1+\rho_2^\top W^G_2=0,\label{eq_Jphi_sys_3}\\
\rho_1^\top W^H_1AB^\top+\rho_3^\top( W^G_3+W^H_3BB^\top)-\rho_4^\top W^H_4=0,\label{eq_Jphi_sys_4}\\
\rho_3^\top W^H_3+\rho_4^\top W^G_4=0.\label{eq_Jphi_sys_5}	    
\end{subnumcases}
It is easy to see that $w_i^G\in(0,1)$, $w_i^H\in(0,1)$, $i\in[m]$. Therefore, by \eqref{eq_Jphi_sys_2}, $\rho_1=(W^G_1)^{-1}W^H_2\rho_2$, where $(W^G_1)^{-1}W^H_2$ is a diagonal matrix with positive diagonal elements. Substituting it into \eqref{eq_Jphi_sys_3}, we obtain that
\begin{equation*}\rho^\top_2\left(W^H_2(W^G_1)^{-1}W^H_1+W^G_2\right)=0,\end{equation*}
where $W^H_2(W^G_1)^{-1}W^H_2+W^G_2$ is also a diagonal matrix with positive diagonal elements. It implies that $\rho_2=0$ and hence $\rho_1=0$. Therefore, \eqref{eq_Jphi_sys_4} reduces to the following equations
\begin{equation*}\rho_3^\top( W^G_3+W^H_3BB^\top)-\rho_4^\top W^H_4=0.\end{equation*}
By \eqref{eq_Jphi_sys_5}, we obtain that $\rho_4 = -(W^G_4)^{-1}W^H_3\rho_3$, where $(W^G_4)^{-1}W^H_3$ is a diagonal matrix with positive diagonal elements. Substituting it into the above equations, we obtain that 
\begin{equation*}\rho_3^\top\left(W_3^G+W^H_3BB^\top +W_3^H(W^G_4)^{-1}W^H_4\right)=0.\end{equation*}
Note that $W_3^G+W^H_3BB^\top +W_3^H(W^G_4)^{-1}W^H_4$ is a positive definite matrix. It gives that $\rho_3=0$. Hence $\rho_4=0$. 	
Overall, we obtain that $\rho=0$, implying that the row vectors in $J_ v \Phi_\epsilon(G( v ),H( v ))$ are linearly independent. Hence $J_ v \Phi_\epsilon(G( v ),H( v ))$ is of full row rank. 
Therefore, LICQ holds at each feasible point of \eqref{pb_nlp}.	
\end{proof}

The Lagrange function of \eqref{pb_nlp} is given by
\begin{equation*}L_\epsilon( v ,\lambda)=f( v )-\lambda^\top \Phi_\epsilon(H( v ), G( v )),\end{equation*}
where $\lambda\in \mathbb{R}^m$ is the Lagrange multiplier corresponding to the equality constraints in \eqref{pb_nlp}.
If $\bar  v $ is a local minimizer of \eqref{pb_nlp}, by LICQ at $\bar  v $ in Theorem \ref{thm_licq}, there exists a unique Lagrange multiplier $\bar\lambda$ such that the following KKT condition holds at $(\bar  v ,\bar{\lambda})$, i.e.,
\begin{equation}\label{pb_kkt}
	F_\epsilon(\bar v ,\bar{\lambda})=0,\ {\rm where}\ 
	F_\epsilon(\bar v ,\bar{\lambda}):= 
	\left[\begin{array}{c}
		\nabla_ {\bar v}  L_\epsilon(\bar v,\bar{\lambda})\\
		-\Phi_\epsilon(G(\bar v),H(\bar v))\end{array}\right]\in\mathbb{R}^{2m+1}.
\end{equation} 

\subsection {Second order sufficient condition}
We focus on the second-order sufficient condition of \eqref{pb_nlp} at $(\bar  v ,\bar{\lambda})$ in this subsection. First, we consider the critical cone $\mathcal C(\bar v ,\bar\lambda)$ at the solution $(\bar v ,\bar\lambda)$ of the KKT system.
By the definition of critical cone \cite[chapter 12, page 330]{nocedal2006numerial}, it holds that 
\begin{equation*}
\mathcal C(\bar v ,\bar\lambda)=\{d \in \mathbb{R}^{m+1} \mid d = (d^C;d^\zeta;d^z;d^\alpha;d^\xi)\mid J_ {\bar v} \Phi_\epsilon(G( \bar v ),H( \bar v )) d=0\}.
\end{equation*}
We have the following result of $d$ in detail.

\begin{proposition}\label{prop-critical}
For any $d\in\mathcal C(\bar{ v },\bar{\lambda})\backslash \{0\}$, $d =Ud^C$ where
\begin{equation} \label{eq_u} \footnotesize
    U=\left[\begin{array}{c}
		1\\
		U^\zeta\\
		U^z\\
		U^\alpha\\
		U^\xi
	\end{array}\right], 
{\rm \ and\ }
\begin{cases}
U^\zeta=-(W^H_2)^{-1}W^G_2\left(W^G_1(W^H_2)^{-1}W^G_2+W^H_1\right)^{-1}W^H_1AB^\top U^\alpha\\
U^z= -\left(W^G_1(W^H_2)^{-1}W^G_2+W^H_1\right)^{-1}W^H_1AB^\top U^\alpha,\\
U^\alpha=\left(W^H_4+W^G_4\left((W^H_3)^{-1}W^G_3+BB^\top\right)\right)^{-1}W^H_4 \mathbf{1},\\
U^\xi=-(W^H_3)^{-1}(W^G_3+W^H_3BB^\top)U^\alpha.
\end{cases}\end{equation}
\end{proposition}

\begin{proof}
By Proposition \ref{prop-JPhi} and \eqref{eq_Jphi}, it holds that 
\begin{equation*}\small
J_ {\bar v} \Phi_\epsilon(G( \bar v ),H( \bar v ))d=
\left[\begin{array}{ccccc}
    {0}&W^G_1&W^H_1&W^H_1AB^\top&\mathbf{0}\\
    {0}&-W^H_2&W^G_2&\mathbf{0}&\mathbf{0}\\
    {0}&\mathbf{0}&\mathbf{0}&W^G_3+W^H_3BB^\top&W^H_3\\
    W^H_4\mathbf{1}&\mathbf{0}&\mathbf{0}&-W^H_4&W^G_4\\
\end{array}\right]
\left[\begin{array}{c}
    d^C\\
    d^\zeta\\
    d^z\\
    d^\alpha\\
    d^\xi
\end{array}\right]=0.
\end{equation*}
Solving the system and noting that $W^G_i,\ W^H_i, \ i\in[4]$ are all diagonal matricies with diagonal elments in $(0,1)$, we then obtain that
\begin{equation*}\begin{cases}
d^\zeta=-(W^H_2)^{-1}W^G_2\left(W^G_1(W^H_2)^{-1}W^G_2+W^H_1\right)^{-1}W^H_1AB^\top d^\alpha, \\
d^z=-\left(W^G_1(W^H_2)^{-1}W^G_2+W^H_1\right)^{-1}W^H_1AB^\top d^\alpha,\\
d^\alpha= \left(W^H_4+W^G_4\left((W^H_3)^{-1}W^G_3+BB^\top\right)\right)^{-1}W^H_4 \mathbf{1} d^C,\\
d^\xi= -(W^H_3)^{-1}(W^G_3+W^H_3BB^\top)d^\alpha .
\end{cases}\end{equation*}
By the definition of $U^\zeta, \ U^z,\ U^\alpha$ and $U^\xi$ in \eqref{eq_u}, we obtain the result.
\end{proof}

Next, we have the following characterization of $\nabla_{vv}^2 L_\epsilon(v,\lambda)$, which is an important part of the second order sufficient condition.

\begin{proposition}\label{prop-hessian}
$\nabla^2_{ v  v }L_\epsilon( v ,\lambda)$ takes the following form
\begin{equation} \footnotesize\label{eq-hessian}
\nabla^2_{ v  v }L_\epsilon( v ,\lambda)=
(L^G)^\top M^GL^G+(L^H)^\top M^HL^H+(L^G)^\top M^{GH}L^H+(L^H)^\top M^{GH}L^G,	
\end{equation}
where
\begin{equation*}\small
	\begin{aligned}
		M^G\ &={\rm Diag}\left( 
		\frac{\lambda_1(H_1( v )^2+\epsilon^2)}{(\sqrt{G_i( v )^2+H_i( v )^2+\epsilon^2})^3},\cdots,
		\frac{\lambda_m(H_m( v )^2+\epsilon^2)}{(\sqrt{G_m( v )^2+H_m( v )^2+\epsilon^2})^3}
		\right),\\
		M^H\ &={\rm Diag}\left( 
		\frac{\lambda_1(G_1( v )^2+\epsilon^2)}{(\sqrt{G_i( v )^2+H_i( v )^2+\epsilon^2})^3},\cdots,
		\frac{\lambda_m(G_m( v )^2+\epsilon^2)}{(\sqrt{G_m( v )^2+H_m( v )^2+\epsilon^2})^3}
		\right),\\
		M^{GH}&={\rm -Diag}\left( 
		\frac{\lambda_1H_1( v )G_1( v )}{(\sqrt{G_i( v )^2+H_i( v )^2+\epsilon^2})^3},\cdots,
		\frac{\lambda_1H_1( v )G_1( v )}{(\sqrt{G_m( v )^2+H_m( v )^2+\epsilon^2})^3}
		\right).\\
	\end{aligned}
\end{equation*}

\end{proposition}

\begin{proof}
By the definition of Lagrange function, it holds that
\begin{equation*}
    \nabla_ v  L_\epsilon ( v ,\lambda)=\nabla_ v  f( v )- \sum_{i=1}^m\lambda_i\nabla_ v \phi_\epsilon(G_i( v ),H_i( v )).
\end{equation*}
Together with the fact that $f( v )$ is a linear function, we obtain that
\begin{equation}\label{eq-def-hessian-2}
    \nabla^2_{ v  v } L_\epsilon( v ,\lambda) = -\sum_{i=1}^m\lambda_i\nabla^2_{ v  v }\phi_\epsilon(G_i( v ),H_i( v )). \\
\end{equation}
According to the definition of $\phi_\epsilon(a,b)$, it holds that
\begin{equation*}\frac{\partial^2\phi_\epsilon(a,b)}{\partial a\partial b}=\frac{ab}{(a^2+b^2+\epsilon^2)^{\frac32}},\ \frac{\partial^2\phi_\epsilon(a,b)}{\partial a^2}=\frac{-(b^2+\epsilon^2)}{(a^2+b^2+\epsilon^2)^{\frac32}}.\end{equation*}
Therefore, by Proposition \ref{prop-JPhi}, $\nabla^2_{ v  v }\phi_\epsilon(G_i( v ),H_i( v ))$ in \eqref{eq-def-hessian-2} is given by ($i\in[m]$)
\begin{equation}\label{eq-hessian-1}\scriptsize
\nabla^2_{ v  v }\phi_\epsilon(G_i( v ),H_i( v ))
=\frac{\partial}{\partial  v }\left(W^G_iL^G(i,:)+W^H_iL^H(i,:)\right)
=L^G(i,:)^\top\frac{\partial W^G_i}{\partial  v } + L^H(i,:)^\top\frac{\partial W^H_i}{\partial  v }.
\end{equation}
where
\begin{eqnarray*}\frac{\partial W^G_i}{\partial  v }
&=& \frac{\partial W^G_i}{\partial G_i(v)}\frac{\partial G_i(v)}{\partial  v }+\frac{\partial W^G_i}{\partial H_i(v)}\frac{\partial H_i(v)}{\partial  v }
=\frac{\partial W^G_i}{\partial G_i(v)}L^G(i,:) + \frac{\partial W^G_i}{\partial H_i(v)}L^H(i,:) \\
&=&  -\frac{ H_i( v )^2+\epsilon^2}{(G_i( v )^2+H_i( v )^2+\epsilon^2)^\frac32}L^G(i,:) + \frac{ G_i( v )H_i( v )}{(G_i( v )^2+H_i( v )^2+\epsilon^2)^\frac32}L^H(i,:),
\end{eqnarray*}
and similarly,
\begin{eqnarray*}\frac{\partial W^H_i}{\partial  v }&=&  -\frac{ G_i( v )^2+\epsilon^2}{(G_i( v )^2+H_i( v )^2+\epsilon^2)^\frac32}L^H(i,:) + \frac{ G_i( v )H_i( v )}{(G_i( v )^2+H_i( v )^2+\epsilon^2)^\frac32}L^G(i,:).
\end{eqnarray*}
Therefore, \eqref{eq-hessian-1} can be transformed as follows:
\begin{equation*}
\begin{aligned}
&\nabla^2_{ v  v }\phi_\epsilon(G_i( v ),H_i( v ))\\
&={\mathsmaller{ \mathsmaller{-\dfrac{ H_i( v )^2+\epsilon^2}{(G_i( v )^2+H_i( v )^2+\epsilon^2)^\frac32} (L^G(i,:))^\top L^G(i,:) + \dfrac{ G_i( v )H_i( v )}{(G_i( v )^2+H_i( v )^2+\epsilon^2)^\frac32}(L^G(i,:))^\top L^H(i,:)}}}\\
&{\mathsmaller{ \mathsmaller{-\dfrac{ G_i( v )^2+\epsilon^2}{(G_i( v )^2+H_i( v )^2+\epsilon^2)^\frac32} (L^H(i,:))^\top L^H(i,:)+ \dfrac{ G_i( v )H_i( v )}{(G_i( v )^2+H_i( v )^2+\epsilon^2)^\frac32} (L^H(i,:))^\top L^G(i,:).}}}
\end{aligned}
\end{equation*}
Substituting it into \eqref{eq-def-hessian-2}, we get the result in \eqref{eq-hessian}. 
\end{proof}

Before giving the final theorem of second order sufficient condition, we need an assumption below.

\begin{assumption}\label{ass_2}
	Let $U\in\mathbb{R}^m$ be defined as in Proposition \ref{prop-critical}. Assume that for $(\bar{ v }, \bar{\lambda})$ satisfying the KKT system $F_\epsilon( v ,\lambda)=0$, the following holds
	\begin{equation*}(U^G)^\top M^G U^G + (U^H)^\top M^H U^H + 2 (U^G)^\top M^{GH} U^H > 0,\end{equation*}where 
	\begin{equation*}U^G = \left[\begin{array}{c}U^\zeta\\ U^z\\ U^\alpha\\ U^\xi \end{array}\right],\ U^H = \left[\begin{array}{c} U^z+AB^\top U^\alpha \\ -U^\zeta\\ BB^\top U^\alpha+U^\xi \\ \mathbf{1}-U^\alpha \end{array}\right].\end{equation*}
\end{assumption}

Under this assumption, we give the second order sufficient condition of \eqref{pb_nlp}.

\begin{theorem}
	\label{thm-ssosc}
	For $(\bar{ v }, \bar{\lambda})$ satisfying the KKT system $F_\epsilon( v ,\lambda)=0$, let Assumption \ref{ass_2} hold. The second order sufficient condition holds at $(\bar{ v }, \bar{\lambda})$. That is, 
	for any $d\in\mathcal C(\bar{ v }, \bar{\lambda})\backslash\{0\}$,
	$d^\top \nabla^2_{\bar{ v } \bar{ v } }L_\epsilon(\bar{ v } ,\bar{\lambda}) d>0$.
    Therefore, $\bar{ v }$ is a strict local minimizer of \eqref{pb_nlp}.
\end{theorem}

\begin{proof}
By Proposition \ref{prop-critical} and Proposition \ref{prop-hessian}, for any 
$d\in\mathcal C(\bar{ v }, \bar{\lambda})\backslash\{0\}$, it holds that (recall \eqref{def_LG_LH_bH})
\begin{equation*}
\begin{aligned}
&d^\top \nabla^2_{\bar v\bar v}L_\epsilon(\bar v,\bar\lambda)d\\
&= (d^C)^2 U^\top \nabla^2_{\bar v\bar v} L_\epsilon(\bar v,\bar\lambda) U \\
&={\mathsmaller{ (d^C)^2 U^\top \big( (L^G)^\top M^G L^G+(L^H)^\top M^H L^H
+ (L^G)^\top M^{GH}L^H+(L^H)^\top M^{GH}L^G \big) U}}\\
&={\mathsmaller{ \mathsmaller{ (d^C)^2 \big( (L^G U)^\top M^G (L^G U) + (L^H U)^\top M^H (L^H U)
+ (L^G U)^\top M^{GH}(L^H U) + (L^H U)^\top M^{GH}(L^G U) \big)}}}\\
&= (d^C)^2 \big( (U^G)^\top M^G U^G + (U^H)^\top M^H U^H + 2 (U^G)^\top M^{GH} U^H \big)\\
&>0 .
\end{aligned}
\end{equation*}


The last inequality is due to Assumption \ref{ass_2}. The proof is finished.
\end{proof}

\subsection{Damped Newton method for \eqref{pb_nlp}}

To solve \eqref{pb_nlp}, one can use software packages, such as SNOPT in \cite{Li_05,Li_06}. Noting the sparsity of $L^G, L^H$ defined by \eqref{def_LG_LH_bH}, combined with the LICQ and second-order properties discussed in Theorem \ref{thm_licq} and Theorem \ref{thm-ssosc}, in our paper we apply damped Newton method proposed in \cite{zhang2013perturbation} to solve \eqref{pb_kkt}, which can make the computing process more targeted and efficient.

Let $g_\epsilon(r)=\frac12\|F_\epsilon(r)\|_2^2$, where $r=( v ,\lambda)\in\mathbb{R}^{2m+1}$. The damped Newton method fills into the traditional line-search update: $r^{k+1}=r^k+s_kd^k$, and $d^k$ is the search direction obtained by solving a linear equation \eqref{eq_Fr} in an inexact way (using for example, BICGSTAB). The details of damped Newton method is given in Algorithm \ref{algo2}. 

\begin{algorithm}[htbp]
\caption{Damped Newton Method for \eqref{pb_nlp}}\label{algo2}
\begin{algorithmic}[1]

\State Choose initial point $r^0$ and parameter $\sigma\in(0,\frac12)$, $\rho\in(0,1)$. Set $k:=0$.

\State \textbf{while} $\|F_\epsilon(r^k)\|_2 \neq 0$, \textbf{do}

\State \quad Find a solution $d^k$ of the linear system by BICG-STAB
\begin{equation} \quad
\label{eq_Fr}
F_\epsilon(r^k)+\Gamma^kd^k=0, 
\end{equation}
\quad where $\Gamma^k=J_rF_\epsilon(r^k)$. 

\State \quad Let $s_k=\rho^{i_k}$, where $i_k$ is the smallest non-negative integer satisfying
\begin{equation*} \quad
g_\epsilon(r^k+\rho^{i_k}d^k)\le g_\epsilon(r^k)+\sigma\rho^{i_k}\nabla g_\epsilon(r^k)^\top d^k,
\end{equation*}
\quad where $g_\epsilon(r)=\frac12\|F_\epsilon(r)\|_2^2$.

\State \quad Let $r^{k+1}=r^k+s_kd^k$, and $k:=k+1$.

\State \textbf{end while}

\State \textbf{Return} the final iterate $r_{\rm opt}=r^k$.

\end{algorithmic}
\end{algorithm}

To make the linear equation \eqref{eq_Fr} in step 3 well-defined, we need to investigate the nonsingularity of $J_r F_\epsilon(\bar r)$, where $\bar r=(\bar v ,\bar\lambda)\in\mathbb{R}^{2m+1}$. Note that $J_r F_\epsilon(\bar r)$ takes the following form
\begin{equation} \label{eq-JF}
	J_r F_\epsilon(\bar r)=\left[\begin{array}{cc}
		\nabla^2_{ v  v }L_\epsilon(\bar v ,\bar\lambda)& -J_{ v }\Phi_\epsilon(G(\bar v ),H(\bar v ))^\top\\
		-J_{ v }\Phi_\epsilon(G(\bar v ),H(\bar v ))&\mathbf{0}	
		\end{array}\right].
\end{equation}

We have the following result.

\begin{theorem}\label{thm-nonsingularity}For $\bar{r}$ satisfying the KKT system $F_\epsilon(\bar r)=0$, let Assumption \ref{ass_2} holds. Then $J_r F_\epsilon(\bar r)$ is nonsingular.
\end{theorem}

\begin{proof}
Assume that $d=(d^ v ;d^\lambda)$ satisfy $J_r F_\epsilon(\bar r)d=0$. By \eqref{eq-JF}, we have
\begin{subnumcases}{}
    \nabla^2_{ v  v }L_\epsilon(\bar v ,\bar\lambda)d^ v - J_{ v }\Phi_\epsilon(G(\bar v ),H(\bar v ))^\top d^\lambda=0,\label{eq-equality-1}\\
    J_{ v }\Phi_\epsilon(G(\bar v ),H(\bar v ))d^ v =0.\label{eq-equality-2}
\end{subnumcases}
Multiplying $(d^ v )^\top$ from the left in \eqref{eq-equality-1}, we obtain that 
\begin{equation*}(d^ v )^\top\nabla^2_{ v  v }L_\epsilon(\bar v ,\bar\lambda)d^ v  - (d^ v )^\top J_{ v }\Phi_\epsilon(G(\bar v ),H(\bar v ))^\top d^\lambda=0.\end{equation*} 
Meanwhile, \eqref{eq-equality-2} can be simplified as 
\begin{equation} \label{eq-1}
    (d^ v )^\top\nabla^2_{ v  v } L_\epsilon(\bar v ,\bar\lambda)d^ v =0.
\end{equation}
By \eqref{eq-equality-2}, we can see that $d^ v \in \mathcal C(\bar{ v },\bar{\lambda})$. Together with Assumption \ref{ass_2}, \eqref{eq-1} implies that $d^ v =0$.
\eqref{eq-equality-1} hence reduces to $J_{ v }\Phi_\epsilon(G(\bar v ),H(\bar v ))^\top d^\lambda=0.$ By the full row rank of  $J_{ v }\Phi_\epsilon(G(\bar v ),H(\bar v ))$, $d^\lambda=0$. Therefore, $J_r F_\epsilon(\bar r)$ is nonsingular.
\end{proof}

By the nonsingularity of $J_r F_\epsilon(\bar r)$, We now state the global convergence and local quadratic convergence rate of Algorithm \ref{algo1} in the following theorem, which could be obtained directly from \textit{Theorem 5.5} in \cite{zhang2013perturbation}.

\begin{theorem}\label{thm-convergence}
Let ${r^k}$ be generated by Algorithm \ref{algo1}. Suppose Assumption \ref{ass_2} holds at the accumulation point $\bar{r}$ of ${r^k}$, then 
(a) $\nabla g_\epsilon(\bar{r})=0$.
(b) If the accumulation point $\bar{r}$ satisfies $F_\epsilon(\bar r)=0$, then ${r^k}$ converge to $\bar{r}$ Q-quadratically.
\end{theorem}

\begin{proof}
    This results follow directly by noting that $J_r F_\epsilon(\bar r)$ is nonsingular in Theorem \ref{thm-nonsingularity}.
\end{proof}

To close this section, we designed the damped Newton method to solve the subproblem \eqref{pb_nlp} by its KKT conditions \eqref{pb_kkt}. Theorem \ref{thm_C} guarantees SDNM converges to an C-stationary point under MPEC-LICQ, and Theorem \ref{thm-convergence} implies that the damped Newton method enjoys a quadratic convergence rate under Assumption \ref{ass_2}.

\section{Numerical Experiments}
In this part, we will conduct extensive numerical tests to verify the efficiency of SDNM.
First, we present the cross-validation algorithm for selecting the hyperparameter $C$ in SVC \cite{Li_05}, as shown in Algorithm \ref{algo3}. The numerical tests are conducted in Matlab R2018b on a Windows 11 Lenovo Laptop with an 13th Gen Intel(R) Core(TM) i5-13500H 2.60 GHz and 32 GB of RAM. All the datasets are collected from the LIBSVM library\footnote{https://www.csie.ntu.edu.tw/cjlin/libsvmtools/datasets/}. 
The data descriptions are shown in Table 1.

\begin{algorithm}[H]
	\caption{The Cross-Validation Algorithm}\label{algo3}
	\begin{algorithmic}[1]
		\vskip 1mm
		\State
		Given $T$, split the data set into a subset $\Omega$ with $p_1$ points and a hold-out test set $\Theta$ with $p_2$ points. The set $\Omega$ is equally partitioned into $T$ pairwise disjoint subsets, one for each fold.
		\vskip 1mm
		\State 
		$\textbf{Find}$ an optimal hyperparameter $C$ by the damped Newton method in Algorithm \ref{algo1}.
		\vskip 1mm
		\State
		$\textbf{Post-processing procedure.}$ The regularization hyperparameter $\hat{C}$ is rescaled by a factor $T/(T-1)$. This gives the final classifier $\hat{w}$.
		\vskip 1mm
	\end{algorithmic}
\end{algorithm}

\begin{table}[H]\label{tab_Descriptions}
\renewcommand{\arraystretch}{1.1}
\setlength{\abovecaptionskip}{0cm}
\caption{\centering{Descriptions of datasets}}
\begin{center}
\begin{tabular}{ c p{7em} r r r | c p{4.5em} r r r} 
\specialrule{0.1em}{0pt}{0pt}																			
No.	&	Dataset	&	$p_1$	&	$p_2$	&	$n$	&	No.	&	Data set	&	$p_1$	&	$p_2$	&	$n$	\\
\specialrule{0.1em}{0pt}{0pt}																			
1	&	diabetes	&	300	&	468	&	8	&	7	&	phishing	&	300	&	1700	&	68	\\
2	&	heart	&	150	&	120	&	13	&	8	&	w1a	&	450	&	750	&	300	\\
3	&	australian	&	300	&	390	&	14	&	9	&	w2a	&	450	&	750	&	300	\\
4	&	german.number	&	300	&	700	&	24	&	10	&	w3a	&	450	&	750	&	300	\\
5	&	ionosphere	&	240	&	111	&	34	&	11	&	w4a	&	450	&	750	&	300	\\
6	&	sonar	&	150	&	58	&	60	&	12	&	w5a	&	450	&	750	&	300	\\
\specialrule{0.1em}{0pt}{0pt}																			
\end{tabular}
\end{center}
\end{table}

\subsection{The performance of SDNM}
In this subsection, we discuss about SDNM from the following two aspects:
(i) The decrease of $\|F_\epsilon(r^k)\|_2$;
(ii) Verification of Assumption \ref{ass_2}.

\subsubsection{The decrease of $\|F_\epsilon(r^k)\|_2$}
Figure \ref{Fig. RE} demonstrates the typical decrease of $\|F_\epsilon(r^k)\|_2$ for some data sets, where one can indeed observe the quadratic convergence rate of SDNM. Each $\star$ indicates the end of an inner loop.

\begin{figure}[H]
	\centering
	\begin{minipage}{0.45\linewidth}
		\centering
		\includegraphics[width=1\linewidth]{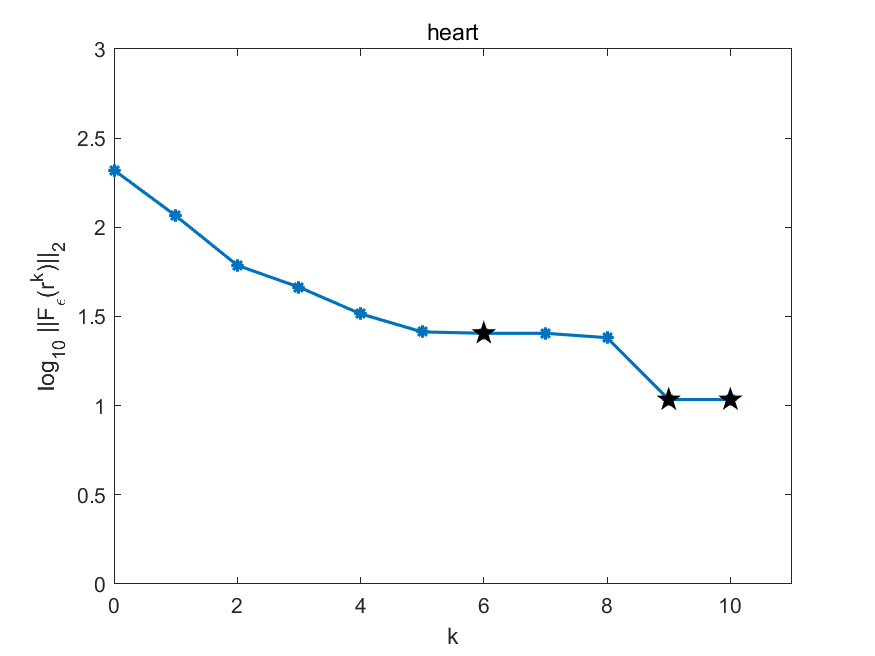}
	\end{minipage}
	\begin{minipage}{0.45\linewidth}
		\centering
		\includegraphics[width=1\linewidth]{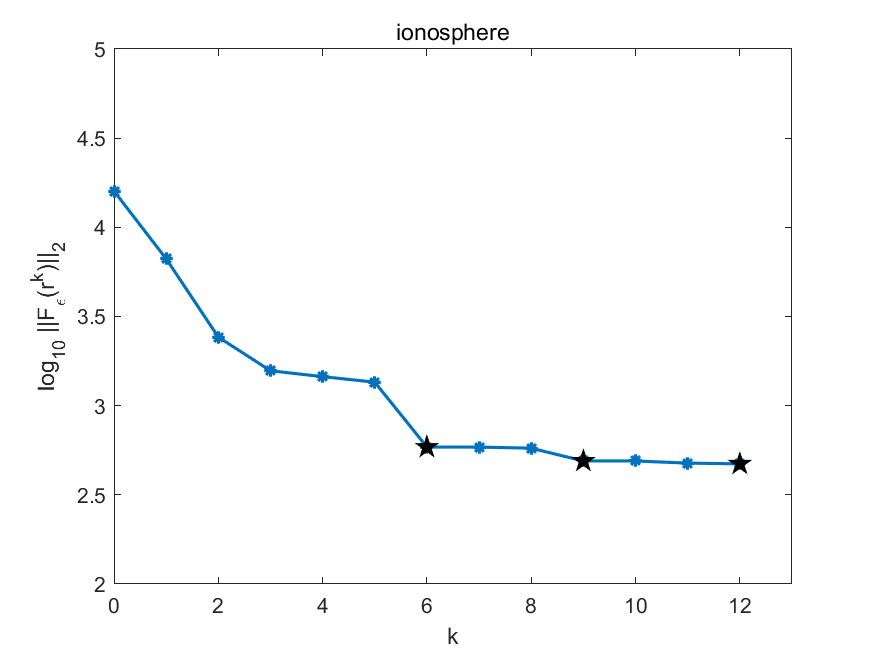}
	\end{minipage}
	\begin{minipage}{0.45\linewidth}
		\centering
		\includegraphics[width=1\linewidth]{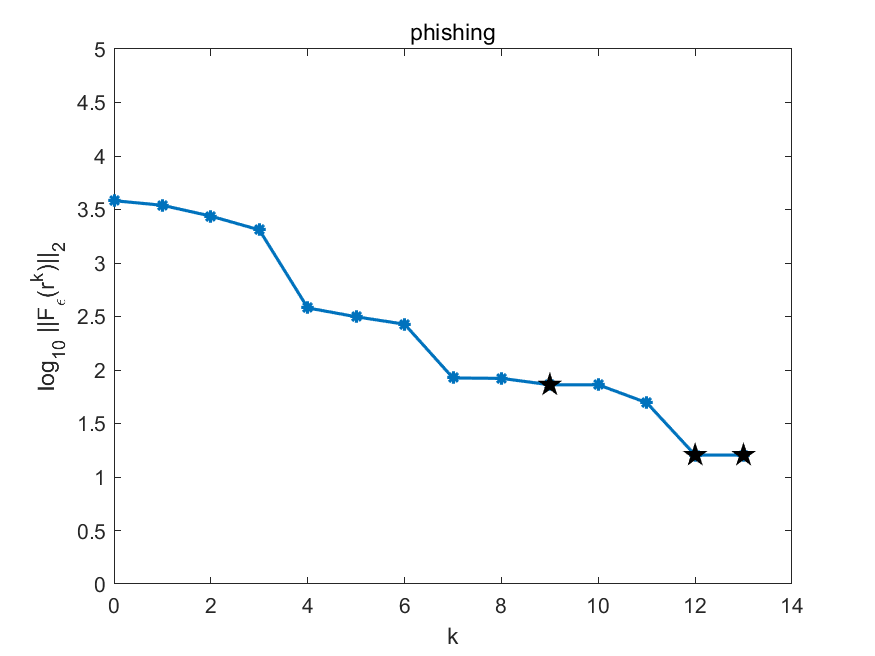}
	\end{minipage}
	\begin{minipage}{0.45\linewidth}
		\centering
		\includegraphics[width=1\linewidth]{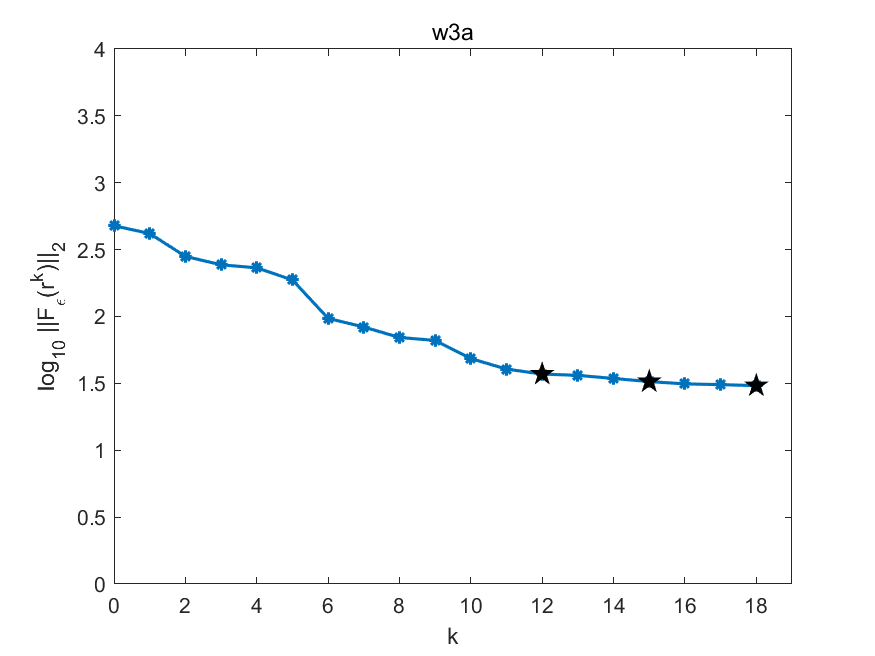}
	\end{minipage}
	\caption{\centering $\log_{10} \|F_\epsilon(r^k)\|_2$ along iterations by SDNM} 
	\label{Fig. RE} 
\end{figure}

\subsubsection{Verification of Assumption \ref{ass_2}}
One can observe from Table 2 that for most data set, the outer loop reaches the stopping criteria within 30 iterations by SDNM. We define $A_2:=\bar{r}^\top_v \nabla^2_{ v  v }L_\epsilon( \bar{v} ,\bar{\lambda}) \bar{r}_v $, and it can be seen that $A_2$ is always positive, which satisfies Assumption \ref{ass_2}, implying that for each $\epsilon$, the damped Newton method provides strict local minimizers of \eqref{pb_nlp}. 


\begin{table}[H]\label{tab_a2}
\renewcommand{\arraystretch}{1.1}
\setlength{\abovecaptionskip}{0cm}
\caption{\centering{Verification of Assumption \ref{ass_2}}}
\begin{center}
\resizebox{1\textwidth}{!}{
\begin{tabular}{ c p{7em} r r r | c p{4em} r r r} 
\specialrule{0.1em}{0pt}{0pt}																			
No.	&	Data set	&	$C\ $	&	$A_2\ $	&	$(		k		,	iter	)$	&	No.	&	Data set	&	$C\ $	&	$A_2\ $	&	$(		k		,	iter	)$	\\
\specialrule{0.1em}{0pt}{0pt}																			
1	&	diabetes	&	2.93 	&	1483.76 	&	$(	\mathbf{	13	}	,	1064	)$	&	7	&	phishing	&	1.77 	&	31.81 	&	$(	\mathbf{	13	}	,	730	)$	\\
2	&	heart	&	0.50 	&	6.49 	&	$(	\mathbf{	10	}	,	1123	)$	&	8	&	w1a	&	0.03 	&	43.86 	&	$(	\mathbf{	10	}	,	458	)$	\\
3	&	australian	&	3.13 	&	30.65 	&	$(	\mathbf{	10	}	,	1248	)$	&	9	&	w2a	&	5.20 	&	488.09 	&	$(	\mathbf{	11	}	,	582	)$	\\
4	&	german.number	&	15.00 	&	997.64 	&	$(	\mathbf{	25	}	,	619	)$	&	10	&	w3a	&	0.08 	&	87.95 	&	$(	\mathbf{	18	}	,	2119	)$	\\
5	&	ionosphere	&	8.55 	&	889.06 	&	$(	\mathbf{	12	}	,	915	)$	&	11	&	w4a	&	2.31 	&	4761.02 	&	$(	\mathbf{	14	}	,	1391	)$	\\
6	&	sonar	&	7.82 	&	428.05 	&	$(	\mathbf{	15	}	,	939	)$	&	12	&	w5a	&	0.07 	&	58.98 	&	$(	\mathbf{	20	}	,	1871	)$ \\
\specialrule{0.1em}{0pt}{0pt}																			
\end{tabular}}
\end{center}
\end{table}

\subsection{Comparison with other methods}

We set the initial values of the parameters of SDNM as $C=1,\ \rho=0.5,\ \sigma=1\times{10}^{-4}$. We compare our method with the following two methods:
\begin{itemize}
\item \textbf{\textit{fmincon}}, solve \eqref{pb_mpec} directly with a Matlab built-in function \textit{fmincon}.
\item \textbf{SGRM}, solve \eqref{pb_mpec} by Scholtes global relaxation method proposed in \cite{Li_05}. SNOPT \cite{snopt} is used to solve the subproblem in each iteration $k$ (given $s_k>0$):
\begin{equation}\label{pb_mpecr}\tag{\text{NLP$_{s_k}$}}
\begin{array}{cl}
\min\limits_{v\in\mathbb{R}^{m+1}} & F(v)\\
\mathrm{s.t.}& G_{i}(v)\geq0,\ \forall i=1,\cdots,m,\\
&H_{i}(v)\geq0,\ \forall i=1,\cdots,m,\\
&G_{i}(v)H_{i}(v)\leq s_{k},\ \forall i=1,\cdots,m. 
\end{array}
\end{equation}
\end{itemize}

The initial points in \textit{fmincon} and SGRM are the same as those in SDNM. 



We compare the aforementioned methods in the following three aspects:
\begin{itemize}
\item [(a)] CPU time, which is the time of hyperparameter optimization solving by aforementioned methods with the given data set in Table 1. 
\item [(b)] Test error ($E_{te}\left(\%\right) $) as defined by $E_{te}=\frac{1}{p_2} \sum _{(x_i,y_i)\in\Theta} \lvert {\rm sign} \left( x_i^\top w^* \right) -y_i \rvert$,
where $\Theta$ is the test set.
\item[(c)] Cross validation error ($E_{CV}\left(\%\right)$) as defined in the objective function of problem \eqref{pb_bilevel}.
\end{itemize}

The results are reported in Table 3, where we mark the winners of $t$, $E_{te}\left(\%\right)$ and $E_{CV}\left(\%\right)$ in bold. ($\beta_1$,$\beta_2$) represents the number of dimensions of variable and the number of constraints in the optimization problem of each methods. As shown in Table 3, SDNM performs better than other methods in terms of test error ($E_{te}(\%)$) and CV error ($E_{CV}\left(\%\right)$), implying that this method has superior generalization performance on classification prediction and solving the over-fitting problem.

In terms of the CPU time summarized in Figure \ref{Fig. t}, SDNM takes the shortest time among the three methods in all data sets. In order to demonstrate how much the SDNM accelerates the optimization process, in Table 3, we also present the time ratio, which shows that SNDM is always faster than the better one of the existing approaches (\textit{fmincon} and SGRM), and is about twice to 20 times as fast as the slower one. In particular, SGRM does well in small-scale data sets (No.1-7), and \textit{fmincon} is competitive in dealing with large scale data sets (No.8-12), but SDNM proposed in this paper is the winner of all of the datasets. For example, dataset w1a (No.8) with 450 cross-validation data points and 300 features can return a solution of hyperparameter in 10 seconds, which is over 20 times faster than SGRM and over 3 times faster than \textit{fmincon}.

\begin{table}[htbp]
\caption{\centering{Computational results for $T$ = 3}} \label{tab_results} 
\begin{center}
\resizebox{0.9\textwidth}{!}{
\begin{tabular}{ c l l c c c c c c} 
\specialrule{0.1em}{0pt}{0pt}																											
No.	&	Dataset	&	Method	&	$(	\beta_1	,	\beta_2	)$	&	$C$	&	CPU time	&	time ratio	&		$E_{te}$		&		$E_{CV}$		\\
\specialrule{0.1em}{0pt}{0pt}																											
1	&	diabetes	&	\textit{fmincon}	&	$(	1801	,	5400	)$	&	9.80 	&		16.5 		&	1.99 	&	$\mathbf{	23.5 	}$	&	$\mathbf{	22.7 	}$	\\
	&		&	SGRM	&	$(	1801	,	5400	)$	&	1.50 	&		8.7 		&	1.05 	&		23.7 		&		23.7 		\\
	&		&	SDNM	&	$(	1801	,	1800	)$	&	2.93 	&	$\mathbf{	8.3 	}$	&	1.00 	&		24.1 		&		23.0 		\\
\hline																											
2	&	heart	&	\textit{fmincon}	&	$(	901	,	2700	)$	&	11.81 	&		3.4 		&	2.90 	&		18.3 		&	$\mathbf{	16.0 	}$	\\
	&		&	SGRM	&	$(	901	,	2700	)$	&	1.62 	&		1.7 		&	1.45 	&		17.5 		&		18.0 		\\
	&		&	SDNM	&	$(	901	,	900	)$	&	0.50 	&	$\mathbf{	1.2 	}$	&	1.00 	&	$\mathbf{	15.8 	}$	&		18.0 		\\
\hline																											
3	&	australian	&	\textit{fmincon}	&	$(	1801	,	5400	)$	&	4.83 	&		17.7 		&	2.32 	&	$\mathbf{	14.6 	}$	&	$\mathbf{	13.3 	}$	\\
	&		&	SGRM	&	$(	1801	,	5400	)$	&	1.46 	&		15.9 		&	2.09 	&	$\mathbf{	14.6 	}$	&		14.0 		\\
	&		&	SDNM	&	$(	1801	,	1800	)$	&	3.13 	&	$\mathbf{	7.6 	}$	&	1.00 	&		15.1 		&		13.7 		\\
\hline																											
4	&	german.	&	\textit{fmincon}	&	$(	1801	,	5400	)$	&	25.15 	&		18.3 		&	2.08 	&		25.3 		&	$\mathbf{	26.7 	}$	\\
	&	number	&	SGRM	&	$(	1801	,	5400	)$	&	1.40 	&		12.1 		&	1.38 	&	$\mathbf{	24.7 	}$	&		28.7 		\\
	&		&	SDNM	&	$(	1801	,	1800	)$	&	15.00 	&	$\mathbf{	8.8 	}$	&	1.00 	&		25.3 		&	$\mathbf{	26.7 	}$	\\
\hline																											
5	&	ionosphere	&	\textit{fmincon}	&	$(	1441	,	4320	)$	&	8.76 	&		11.6 		&	2.23 	&	$\mathbf{	2.7 	}$	&		26.7 		\\
	&		&	SGRM	&	$(	1441	,	4320	)$	&	1.50 	&		5.5 		&	1.06 	&		4.5 		&	$\mathbf{	25.4 	}$	\\
	&		&	SDNM	&	$(	1441	,	1440	)$	&	8.55 	&	$\mathbf{	5.2 	}$	&	1.00 	&	$\mathbf{	2.7 	}$	&		26.7 		\\
\hline																											
6	&	sonar	&	\textit{fmincon}	&	$(	901	,	2700	)$	&	13.39 	&		4.4 		&	2.82 	&	$\mathbf{	27.6 	}$	&		30.0 		\\
	&		&	SGRM	&	$(	901	,	2700	)$	&	2.64 	&		2.8 		&	1.77 	&	$\mathbf{	27.6 	}$	&		30.0 		\\
	&		&	SDNM	&	$(	901	,	900	)$	&	7.82 	&	$\mathbf{	1.6 	}$	&	1.00 	&	$\mathbf{	27.6 	}$	&	$\mathbf{	29.3 	}$	\\
\hline																											
7	&	phishing	&	\textit{fmincon}	&	$(	1801	,	5400	)$	&	5.64 	&		18.0 		&	3.08 	&	$\mathbf{	6.4 	}$	&	$\mathbf{	7.0 	}$	\\
	&		&	SGRM	&	$(	1801	,	5400	)$	&	1.54 	&		17.4 		&	2.97 	&	$\mathbf{	6.4 	}$	&		8.3 		\\
	&		&	SDNM	&	$(	1801	,	1800	)$	&	1.77 	&	$\mathbf{	5.9 	}$	&	1.00 	&		6.9 		&		9.0 		\\
\hline																											
8	&	w1a	&	\textit{fmincon}	&	$(	2701	,	8100	)$	&	3.46 	&		35.7 		&	3.93 	&		3.5 		&	$\mathbf{	3.6 	}$	\\
	&		&	SGRM	&	$(	2701	,	8100	)$	&	1.50 	&		231.6 		&	25.52 	&		3.3 		&		3.8 		\\
	&		&	SDNM	&	$(	2701	,	2700	)$	&	0.03 	&	$\mathbf{	9.1 	}$	&	1.00 	&	$\mathbf{	3.1 	}$	&		3.8 		\\
\hline																											
9	&	w2a	&	\textit{fmincon}	&	$(	2701	,	8100	)$	&	3.46 	&		36.1 		&	1.91 	&	$\mathbf{	3.2 	}$	&		3.6 		\\
	&		&	SGRM	&	$(	2701	,	8100	)$	&	1.51 	&		137.1 		&	7.26 	&	$\mathbf{	3.2 	}$	&		3.6 		\\
	&		&	SDNM	&	$(	2701	,	2700	)$	&	5.20 	&	$\mathbf{	18.9 	}$	&	1.00 	&		3.3 		&	$\mathbf{	3.3 	}$	\\
\hline																											
10	&	w3a	&	\textit{fmincon}	&	$(	2701	,	8100	)$	&	4.06 	&		37.1 		&	1.01 	&		2.4 		&		3.3 		\\
	&		&	SGRM	&	$(	2701	,	8100	)$	&	1.50 	&		250.8 		&	6.82 	&		2.4 		&		3.3 		\\
	&		&	SDNM	&	$(	2701	,	2700	)$	&	0.08 	&	$\mathbf{	36.8 	}$	&	1.00 	&	$\mathbf{	2.1 	}$	&	$\mathbf{	2.9 	}$	\\
\hline																											
11	&	w4a	&	\textit{fmincon}	&	$(	2701	,	8100	)$	&	3.64 	&		36.6 		&	1.23 	&	$\mathbf{	3.1 	}$	&	$\mathbf{	3.1 	}$	\\
	&		&	SGRM	&	$(	2701	,	8100	)$	&	1.48 	&		236.0 		&	7.92 	&	$\mathbf{	3.1 	}$	&	$\mathbf{	3.1 	}$	\\
	&		&	SDNM	&	$(	2701	,	2700	)$	&	2.31 	&	$\mathbf{	29.8 	}$	&	1.00 	&	$\mathbf{	3.1 	}$	&	$\mathbf{	3.1 	}$	\\
\hline																											
12	&	w5a	&	\textit{fmincon}	&	$(	2701	,	8100	)$	&	3.66 	&		37.0 		&	1.01 	&		4.1 		&		4.2 		\\
	&		&	SGRM	&	$(	2701	,	8100	)$	&	1.52 	&		257.2 		&	7.02 	&		4.1 		&		4.4 		\\
	&		&	SDNM	&	$(	2701	,	2700	)$	&	0.07 	&	$\mathbf{	36.6 	}$	&	1.00 	&	$\mathbf{	2.4 	}$	&	$\mathbf{	3.1 	}$	\\
\specialrule{0.1em}{0pt}{0pt}										
\end{tabular}}
\end{center}
\end{table}

\begin{figure}[htbp]
	\centering
	\setlength{\abovecaptionskip}{0.1cm}
	\begin{minipage}{0.85\linewidth}
		\centering
		\includegraphics[width=0.95\linewidth]{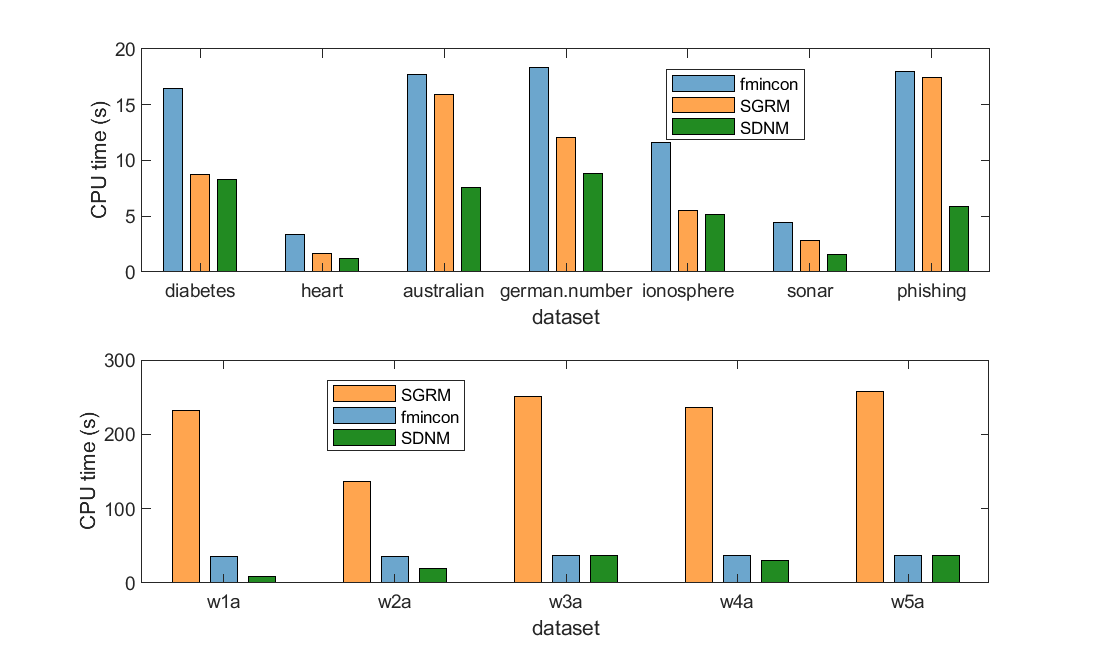}
	\end{minipage}
	\caption{\centering CPU time of three methods} 
	\label{Fig. t} 
\end{figure}

Overall, SDNM demonstrates high efficiency in practice with superior generalization performance on classification prediction.

\section{Conclusion}\label{sec6}

In this paper, we develop an efficient algorithm to solve the \eqref{pb_mpec} derived from the bilevel optimization model for hyperparameter selection for SVC. Specifically, we smoothed the complementarity constraints in the \eqref{pb_mpec} and proposed the smoothing damped Newton method to solve it. 
Under certain assumption, the proposed SDNM converges to C-stationary point under MPEC-LICQ with subproblem enjoys a quadratic convergence rate under proper assumptions.
Finally, extensive numerical results on the datasets from the LIBSVM library verified the efficiency of SDNM over almost all the datasets used in this paper comparing with the other two methods. 
SDNM demonstrates high efficiency in practice with superior generalization performance on classification prediction and solving the over-fitting problem, and is always faster than the existing approaches.

\bibliographystyle{plain}
\bibliography{Reference}

\section*{Conflict of interest}
The authors declare that they have no conflict of interest.

\end{document}